\documentclass[11pt,letterpaper]{amsart}

\pdfoutput=1

\usepackage[T1]{fontenc}
\usepackage{amsmath,amssymb,amsthm}
\usepackage[margin=2.75cm]{geometry}
\usepackage{microtype}
\usepackage[hidelinks]{hyperref}
\usepackage{xcolor}

\allowdisplaybreaks
\newtheorem{theorem}{Theorem}[section]
\newtheorem{lemma}[theorem]{Lemma}
\newtheorem{proposition}[theorem]{Proposition}
\newtheorem{corollary}[theorem]{Corollary}
\theoremstyle{remark}
\newtheorem{remark}[theorem]{Remark}
\newtheorem*{unnumberedremark}{Remark}
\newtheorem{example}[theorem]{Example}

\theoremstyle{plain}
\newtheorem{headthm}{Theorem}

\numberwithin{equation}{section}

\DeclareMathOperator{\Proj}{Proj}
\DeclareMathOperator{\Sym}{Sym}
\DeclareMathOperator{\vol}{vol}
\DeclareMathOperator{\spann}{span}

\newcommand{\m}{\mathfrak m}
\newcommand{\OO}{\mathcal O}
\newcommand{\F}{\mathcal F}
\newcommand{\Kcal}{\mathcal K}
\newcommand{\Mcal}{\mathcal M}
\newcommand{\Ical}{\mathcal I}

\newcommand{\PP}{\mathbb P}
\newcommand{\QQ}{\mathbb Q}
\newcommand{\RR}{\mathbb R}
\newcommand{\ZZ}{\mathbb Z}

\renewcommand{\leq}{\leqslant}
\renewcommand{\geq}{\geqslant}

\begin{document}

\title[Transcendental Hilbert--Kunz Multiplicities]{Transcendental Hilbert--Kunz Multiplicities}

\author[Sudipta Das]{Sudipta Das}
\address{School of Mathematics, Tata Institute of Fundamental Research, Dr. Homi Bhabha Road, Colaba, Mumbai 400005, India}
\email{sudiptad@math.tifr.res.in}

\author[Stephen Landsittel]{Stephen Landsittel}
\address{Institute of Mathematics, Hebrew University of Jerusalem, Givat Ram, Jerusalem 91904, Israel}
\email{stephen.landsittel@mail.huji.ac.il}

\author[Vinh Anh Ph{\d{A}}m]{Vinh Anh Ph{\d{A}}m}
\address{Department of Mathematics, Tulane University, 6823 St. Charles Avenue, New Orleans, LA 70118, USA}
\email{vpham1@tulane.edu}

\dedicatory{Dedicated to Professor Keiichi Watanabe with gratitude and admiration.}

\begin{abstract}
We prove that ordinary Hilbert--Kunz multiplicity can be transcendental.
More precisely, over every uncountable algebraically closed field of
characteristic $p>2$, there exist a normal standard graded domain $S$
and an $S_+$-primary homogeneous ideal $I\subseteq S$ such that
$e_{\rm HK}(IS_{S_+})$ is transcendental.
\end{abstract}


\maketitle

\section{Introduction}

Which real numbers occur as Hilbert--Kunz multiplicities?  Even the existence of a transcendental value has remained open.  We prove that such values occur.

Let $(R,\mathfrak m)$ be a Noetherian local ring of characteristic $p>0$, and let $I\subseteq R$ be an $\mathfrak m$-primary ideal.  For $Q=p^e$, the $Q$th Frobenius power of $I=(f_1,\ldots,f_s)$ is
$$
        I^{[Q]}=(f_1^Q,\ldots,f_s^Q).
$$
Monsky proved that the normalized lengths
$$
        \frac{\lambda_R(R/I^{[Q]})}{Q^{\dim R}}
$$
converge as $Q\to\infty$ \cite{Monsky83}.  Their limit is the Hilbert--Kunz multiplicity
$$
        e_{\rm HK}(I)
        =
        \lim_{e\to\infty}
        \frac{\lambda_R(R/I^{[p^e]})}{p^{e\dim R}}.
$$
Hilbert--Kunz multiplicity is a numerical invariant of singularities in prime characteristic.  It detects regularity in the unmixed case \cite{WY00}, is closely related to tight closure, $F$-singularities, and $F$-signature, and has geometric interpretations involving vector bundles, Frobenius pullbacks, Harder--Narasimhan filtrations, and density functions \cite{Brenner04,Trivedi18}.

Rationality is known in several important settings.  Brenner proved that the Hilbert--Kunz multiplicity of a homogeneous $R_+$-primary ideal is rational in a two-dimensional graded domain of finite type over an algebraically closed field \cite{Brenner04}.  Monsky later proposed a likely counterexample to rationality \cite{Monsky08} and showed, conditional on precise conjectures for colengths associated with a nodal cubic in characteristic two, that algebraic irrational and transcendental values should occur \cite{Monsky09Alg,Monsky09Trans}.  Brenner constructed the first unconditional irrational example in \cite{Bre13}.  A related transcendence result for epsilon multiplicity was recently obtained from divisor volumes \cite{DasLandsittelPhamEpsilon}.  The corresponding question for ordinary Hilbert--Kunz multiplicity remained open.

\begin{headthm}\label{thm:intro-HK}
Let $k$ be an uncountable algebraically closed field of characteristic $p>2$.  There exists a normal standard graded $k$-domain $S$ and an $S_+$-primary homogeneous ideal $I\subseteq S$ such that
$e_{\rm HK}(IS_{S_+})$
is transcendental.
\end{headthm}

To our knowledge, Theorem~\ref{thm:intro-HK} gives the first unconditional example of a transcendental ordinary Hilbert--Kunz multiplicity.  The proof separates the Frobenius count into two ranges.  When $m<Q$, the Frobenius powers of the generators do not
contribute, and the relevant terms lead to a logarithmic integral of divisor volumes. The range $m\geq Q$ cannot be ignored. There the generators do contribute, and the remaining terms are cokernels of
multiplication maps. A priori, the logarithms coming from this second range could cancel the one obtained for $m<Q$.  We compute these cokernels and prove that the two logarithmic contributions have the
same sign.

\subsection{A transcendental generalized Hilbert--Kunz multiplicity}

We first work with generalized Hilbert--Kunz multiplicity.  For an ideal $J$ in a local ring $(R,\mathfrak m)$, set
$$
        e_{\rm gHK}(J)
        =
        \lim_{e\to\infty}
        \frac{\lambda_R\bigl(H^0_{\mathfrak m}(R/J^{[p^e]})\bigr)}{p^{e\dim R}},
$$
whenever the limit exists.  This agrees with ordinary Hilbert--Kunz multiplicity when $J$ is $\mathfrak m$-primary.  In general it measures the finite length defect between a Frobenius power and its saturation.  The underlying function appears already in work of Aberbach \cite{AberbachFSignature}, and the invariant is a special case of the relative Hilbert--Kunz multiplicities considered by Epstein--Yao \cite{EpsteinYao}.  Dao--Smirnov developed its existence and homological theory \cite{DaoSmirnov}.  Dao--Watanabe computed it for several classes of rings \cite{DaoWatanabe}, Brenner--Caminata treated the two dimensional graded case \cite{BrennerCaminata}, and Asgharzadeh studied its relation with the LC condition and further curve cases \cite{AsgharzadehLC}.  Vraciu proved the reduction result used in Section~\ref{sec:ordinary-HK} \cite{Vraciu}.  A broader treatment for $p$-families of ideals appears in \cite{LandsittelDasFamilies}.

\begin{headthm}\label{thm:intro-ghk}
Let $k$ be an uncountable algebraically closed field of characteristic $p>2$.  There exists a normal standard graded $k$-domain $S$ and a homogeneous ideal $J\subseteq S$ such that
$e_{\rm gHK}(JS_{S_+})$
is transcendental.
\end{headthm}

The construction starts with an unnodal Enriques surface $Y$ carrying three base point free genus one pencils $F_0,F_1,F_2$ satisfying
$$
        F_i^2=0,
        \qquad
        F_i\cdot F_j=4
        \quad(i\neq j).
$$
On the real span of these classes, the nef and pseudo-effective cones are the same closed positive cone.  The volume of a class in this span is therefore its self-intersection inside the cone and is zero outside it. Since $F_i^2=0$ and $F_i\cdot F_j=4$ for $i\neq j$, the self-intersection of
$c_0F_0+c_1F_1+c_2F_2$ is $8(c_0c_1+c_0c_2+c_1c_2).$
Because the volume on their span is the positive part of the self-intersection, it is given by an explicit quadratic function of the coefficients.

Consider the split projective bundle
$$
        \pi:X=\PP_Y\bigl(\OO_Y(F_0)\oplus\OO_Y(F_1)\oplus\OO_Y(F_2)\bigr)
        \longrightarrow Y,
$$
and write $\xi=\OO_X(1)$.
Symmetric powers of the split bundle decompose into line bundles indexed by triples of nonnegative integers.  After normalization, these triples fill a triangle, and sums of dimensions of global sections become Riemann sums for an integral of divisor volumes on $Y$. Write $c_i=r+x_i$, where $x_0+x_1+x_2=0$.  This translates the simplex of normalized coefficients so that its barycenter is the origin.  In these coordinates, the condition that $\sum_i c_iF_i$ lie in the positive cone becomes $x_0^2+x_1^2+x_2^2\leq 6r^2.$
Thus the positive cone cuts the simplex of normalized coefficients by a disk centered at its barycenter.  The integral then evaluates to an algebraic number plus a nonzero algebraic multiple of the logarithm of a positive algebraic number.

Set $S_Y=F_0+F_1+F_2$.  For a sufficiently large integer $N$, put
$$
D_N=N\pi^*S_Y,
\qquad
L_N=\xi+D_N,
$$
and define
$$
S_N:=R(X,L_N)=\bigoplus_{m\geq0}H^0(X,mL_N).
$$
Then $L_N-D_N=\xi$.  Let
$$
W:=H^0(X,\OO_X(1))
   =\bigoplus_{i=0}^2 H^0(Y,\OO_Y(F_i)),
$$
so that $\dim_k W=6$.  Multiplication by the canonical section
$s_{D_N}$ of $D_N$ gives an inclusion
$$
W=H^0(X,L_N-D_N)\longrightarrow H^0(X,L_N)=(S_N)_1.
$$
If $\widehat W:=s_{D_N}W$, we define
$$
J_N:=S_N\widehat W\subseteq S_N.
$$
Thus $J_N$ is generated by six elements of
$(S_N)_1=H^0(X,L_N)$.  Each generator has the form
$s_{D_N}w$ with $w\in W$, so their zero divisors all contain $D_N$. Since $W$ is base point free, $D_N$ is their only common fixed component.  For $Q=p^e$, the saturation of $J_N^{[Q]}$ in $S_N$ is the divisorial ideal associated with $QD_N$. Since $J_N^{[Q]}$ is generated in degree $Q$, its degree $m$ part vanishes for $m<Q$.  Hence, for $m<Q$, the degree $m$ component of the saturation quotient is
$$
H^0(X,mL_N-QD_N).
$$
The normalized sum over $m<Q$ is therefore the volume integral described above.

In degree $Q+a$, Subsection~\ref{subsec:split-pencil-bundle} identifies the corresponding component of the saturation quotient with the cokernel $C_{Q,a}$ of multiplication by the $Q$th powers of the pencil sections.  After pushing this map to $Y$, the splitting of $E$ decomposes $C_{Q,a}$ into blocks.  Lemmas~\ref{lem:upper-cokernel-bookkeeping} and~\ref{lem:uniform-enriques-pencil-cohomology} show that every possible term of order $Q^2$ is governed by a group
$$
H^1\bigl(Y,\OO_Y(D_{\beta,a}-2QF_i)\bigr).
$$
The proof of Proposition~\ref{prop:upper-leading-log-small} then divides the normalized exponent space into regions.  The blocks indexed by triples with $\beta_i<Q$ for every $i$ contribute only algebraic terms; see Lemma~\ref{lem:upper-cokernel-bookkeeping} and the first case in the proof of Proposition~\ref{prop:upper-leading-log-small}.  In the regions where the kernel divisor remains nef, the nontrivial blocks have total contribution $o(Q^5)$.  A logarithm arises only when exactly one exponent $\beta_i$ is at least $Q$ and the corresponding divisor $D_{\beta,a}-2QF_i$ leaves the nef cone.  Lemma~\ref{lem:critical-integral-arithmetic} shows that this logarithmic contribution is strictly positive.

Propositions~\ref{prop:exact-lower-log} and
\ref{prop:upper-leading-log-small} show that, for every fixed sufficiently large $N$, the lower and upper logarithmic contributions are both positive. They therefore cannot cancel.  Their respective orders as
$N\to\infty$ are $N^{-1}$ and $N^{-2}$; these estimates compare their sizes but are not used to prove nonvanishing.  Theorem~ \ref{thm:enriques-gHK-criterion} then writes the resulting generalized
Hilbert--Kunz multiplicity as an algebraic number plus a nonzero algebraic linear combination of logarithms of positive algebraic numbers. Therefore Corollary~\ref{cor:Baker-linear-form} gives Theorem~\ref{thm:intro-ghk}.

\subsection{From generalized to ordinary Hilbert--Kunz multiplicity}

To deduce Theorem~\ref{thm:intro-HK} from Theorem~\ref{thm:intro-ghk}, we use Vraciu's reduction \cite{Vraciu}.  The colon ideals that occur in the induction depend on $Q$.  Lemma~\ref{lem:fixed-homogeneous-LC} gives the required linear LC bounds, while Lemmas~\ref{lem:homogeneous-localization} and~\ref{lem:homogeneous-countable-prime-avoidance} allow the auxiliary elements to be chosen homogeneously and independently of $Q$.  Proposition~\ref{prop:vraciu-reduction-enriques} produces finitely many fixed $\mathfrak m$-primary ideals, and Lemma~\ref{lem:homogeneous-contraction-HK} shows that they come from $S_+$-primary homogeneous ideals $I_1,\ldots,I_t\subseteq S$.  Thus there are integers $a_1,\ldots,a_t$ such that
$$
        e_{\rm gHK}(JS_{S_+})
        =
        \sum_{i=1}^t a_i e_{\rm HK}(I_iS_{S_+}).
$$
Since a finite integer linear combination of algebraic numbers is algebraic, the transcendence of the left side forces at least one Hilbert--Kunz multiplicity on the right to be transcendental.

Odd characteristic provides the classical unnodal Enriques geometry used in the construction.  Algebraic closure supplies the required surface and connected pencil fibers, while perfection enters the Frobenius arguments on smooth varieties.  Uncountability is used only for the simultaneous homogeneous prime avoidance in Section~\ref{sec:ordinary-HK}.  The integer $N$ is chosen sufficiently large so that $L_N$ is ample and the full section ring $R(X,L_N)$ is a normal standard graded domain.  It is then held fixed while $Q=p^e$ tends to infinity.  The later expansion as $N\to\infty$ compares two already positive logarithmic terms and does not interchange these limits.

Sections~\ref{prelim} and~\ref{sec:geometric-input} establish the graded and geometric preliminaries.  Section~\ref{sec:enriques-gHK} proves Theorem~\ref{thm:intro-ghk} by the Enriques construction, and Section~\ref{sec:ordinary-HK} proves Theorem~\ref{thm:intro-HK} through Vraciu's reduction.  Appendix~\ref{app:exact-log-computations} contains the exact evaluations of the two logarithmic terms.

\section{Preliminaries}\label{prelim}

The proof uses three descriptions of the same asymptotic quantity, namely graded local cohomology, global sections on a projective variety, and integrals of divisor volumes. The statements below relate these
descriptions and provide two uniform estimates used later, one for section counts in compact families of divisor classes and one for cohomology after Frobenius pullback.

Throughout the paper, $k$ denotes a field, and varieties are integral and projective over $k$ unless stated otherwise.  We use additive notation for divisors and tensor powers of line bundles.  Thus $mL$ means $L^{\otimes m}$, while
$$
        mL-D=L^{\otimes m}\otimes\OO_X(-D).
$$
The section ring of $L$ is
$$
        R(X,L):=\bigoplus_{m\geq0}H^0(X,mL),
$$
and we write $H^0(X,D)$ for $H^0(X,\OO_X(D))$ when $D$ is a Cartier divisor.

\subsection{Length, Saturation, and the Graded Local Dictionary}\label{subsec:graded-length-saturation}

The geometric calculation is carried out in a standard graded ring, whereas Hilbert--Kunz multiplicity is defined after localization at the homogeneous maximal ideal.  Saturation identifies the finite length contribution that connects the two settings.

Let
$$
        S=\bigoplus_{m\geq 0}S_m
$$
be a standard graded domain over $k$, and assume $S_0=k$.  We write
$$
        S_+:=\bigoplus_{m>0}S_m
$$
for the irrelevant ideal.  If $J\subseteq S$ is a homogeneous ideal, its saturation with respect to $S_+$ is
$$
        J^{\rm sat}:=J:S_+^\infty=\bigcup_{i\geq 1}(J:S_+^i).
$$
The zeroth local cohomology module of $S/J$ with support in $S_+$ is defined as
$$
H^0_{S_+}(S/J)
=
\{\,\overline{x}\in S/J\mid S_+^i\overline{x}=0\text{ for some }i\geq 0\,\}.
$$
The standard identity
$H^0_{S_+}(S/J)\cong J^{\rm sat}/J$
holds. For finite length graded modules, localization at $S_+$ preserves length.

\begin{lemma}\label{lem:graded-local-length}
Let $S$ be a standard graded $k$-algebra with $S_0=k$, and let $N$ be a finite length graded $S$-module.  Then
$\lambda_S(N)=\lambda_{S_{S_+}}(N_{S_+})$.
In particular, if $J\subseteq S$ is homogeneous and $J^{\rm sat}/J$ has finite length, then its length can be computed either in the graded ring $S$ or after localizing at $S_+$.
\end{lemma}

\begin{proof}
Choose a finite graded composition series of $N$.  Its simple factors are degree shifts of $S/S_+$.  Localization at $S_+$ preserves the series and sends each factor to the residue field of $S_{S_+}$, so the two lengths are equal.
\end{proof}

Consequently, every finite length module appearing below may be measured either before or after localization at $S_+$.  We perform the geometric computations in the graded ring and localize only when stating the resulting Hilbert--Kunz invariants.

\subsection{Frobenius Powers, Saturation, and Generalized Hilbert--Kunz Multiplicity}\label{subsec:ghk-prelim}

For a homogeneous ideal, generalized Hilbert--Kunz multiplicity measures the sections recovered by saturation but missing from its Frobenius powers.  The graded formulation below identifies the degree range in which those sections are determined entirely by a divisor.

Let $R$ be a Noetherian local ring of prime characteristic $p>0$, with maximal ideal $\m$ and dimension $d$.  For $Q=p^e$ and an ideal $I=(f_1,\dots,f_s)$, the $Q$th Frobenius power of $I$ is
$$
        I^{[Q]}=(f_1^Q,\dots,f_s^Q).
$$
This definition is independent of the chosen generators.  For an arbitrary ideal the normalized sequence need not be known to converge.  We therefore write
$$
\begin{aligned}
e_{\rm gHK}^+(I)
&=
\limsup_{e\to\infty}
\frac{\lambda_R\bigl(H^0_{\m}(R/I^{[p^e]})\bigr)}{p^{ed}},\\
e_{\rm gHK}^-(I)
&=
\liminf_{e\to\infty}
\frac{\lambda_R\bigl(H^0_{\m}(R/I^{[p^e]})\bigr)}{p^{ed}}.
\end{aligned}
$$
When these two numbers agree, their common value is denoted by $e_{\rm gHK}(I)$.  For $\m$-primary ideals this limit exists and coincides with the usual Hilbert--Kunz multiplicity.  For ideals that are not $\m$-primary, it measures a Frobenius saturation defect and is substantially more subtle; see Dao--Smirnov \cite{DaoSmirnov} and Vraciu \cite[Definition~1.2]{Vraciu}.  The existence and structure of related Hilbert--Kunz limits are part of a broad circle of work in prime characteristic, including the existence of $F$-signature \cite{Tuc12} and examples showing that Hilbert--Kunz invariants need not be rational in general \cite{Bre13}.

We use the graded form below throughout the paper.  Let
$$
        S=\bigoplus_{m\geq 0}S_m
$$
be a standard graded domain over a field of characteristic $p>0$, let $S_+=\bigoplus_{m>0}S_m$, and let $J\subseteq S$ be homogeneous.  Analogously to saturation of ideals in local rings, we have
$$
        H^0_{S_+}(S/J^{[Q]})\cong (J^{[Q]})^{\rm sat}/J^{[Q]}.
$$
Lemma~\ref{lem:graded-local-length} therefore identifies $e_{\rm gHK}^{\pm}(JS_{S_+})$ with the limsup and liminf, respectively, of
$$
\frac{\lambda_S\bigl((J^{[Q]})^{\rm sat}/J^{[Q]}\bigr)}{Q^{\dim S}},
\qquad Q=p^e.
$$
Whenever the latter sequence converges, its limit is $e_{\rm gHK}(JS_{S_+})$.

The distinction from ordinary powers is important.  The ideal $J^{[Q]}$ contains only the $Q$th powers of a chosen set of generators; it need not contain all products of total degree $Q$.  Saturation restores the full sheaf generated by those Frobenius powers, and the difference between the two is the quantity that we shall compute.

When the generators come from a globally generated divisor, the saturated sheaf has the form below.

\begin{lemma}\label{lem:frobenius-saturation-by-degree}
Let $k$ be a field of characteristic $p>0$.  Let $X$ be a normal projective variety over $k$, let $L$ be an ample line bundle with standard graded section ring $S=R(X,L)$, let $D$ be an effective Cartier divisor, and fix $c\geq 1$.  Set
$$
        \Mcal=cL-D,
$$
and assume that $\Mcal$ is globally generated.  Let
$s_D\in H^0(X,\OO_X(D))$
be the canonical section.  Multiplication by $s_D$ gives an injective map
$$
H^0(X,\Mcal)=H^0(X,cL-D)
\longrightarrow
H^0(X,cL)=S_c,
\qquad
w\longmapsto s_Dw.
$$
Let $J\subseteq S$ be the homogeneous ideal generated by this image.  Then, for every $Q=p^e$,
$$
\widetilde{J^{[Q]}}=\OO_X(-QD)\subseteq\OO_X,
\qquad
(J^{[Q]})^{\rm sat}=I_{QD}:=\Gamma_*(\OO_X(-QD)),
$$
and
$(J^{[Q]})_m=0\qquad\text{for all }m<cQ$.
\end{lemma}

\begin{proof}
Choose sections $w_1,\ldots,w_\ell$ that span
$H^0(X,\mathcal M)$, where $\mathcal M=cL-D$.  Since $\mathcal M$ is globally generated, these sections generate $\mathcal M$ as a line bundle.  The ideal $J$ is generated in degree $c$ by
$$
s_Dw_1,\ldots,s_Dw_\ell.
$$
Hence $J^{[Q]}$ is generated in degree $cQ$ by
$$
s_D^Qw_1^Q,\ldots,s_D^Qw_\ell^Q.
$$

The sections $w_1^Q,\ldots,w_\ell^Q$ generate
$\mathcal M^{\otimes Q}$.  Indeed, at every point of $X$, one of the $w_i$ is a local generator of $\mathcal M$, and its $Q$th tensor power is a local generator of $\mathcal M^{\otimes Q}$.

We compute the associated ideal sheaf locally.  Let $U\subseteq X$ be an open set on which $\mathcal M$ is trivial, and let $g$ be a local equation for $D$.  Write $w_i=f_i e$, where $e$ is a local generator
of $\mathcal M$.  Since the $w_i$ generate $\mathcal M$, the functions $f_i$ generate $\OO_X(U)$.  Therefore their $Q$th powers also generate $\OO_X(U)$.  On $U$, the image ideal generated by the sections
$s_D^Qw_i^Q$ is
$$
(g^Qf_1^Q,\ldots,g^Qf_\ell^Q) = (g^Q).
$$
This is the local ideal of $QD$.  Hence
$$
\widetilde{J^{[Q]}}=\OO_X(-QD).
$$

Lemma~\ref{lem:full-section-ring-saturation} now gives
$$
(J^{[Q]})^{\rm sat}
=\Gamma_*(\widetilde{J^{[Q]}})
=\Gamma_*(\OO_X(-QD))
=I_{QD}.
$$
Finally, $J^{[Q]}$ is generated in degree $cQ$, so
$(J^{[Q]})_m=0$ for every $m<cQ$.
\end{proof}

Let
$$
        \tau=\inf\{t\geq0\mid tL-D\text{ is pseudo-effective}\}.
$$
The lemma separates the saturation quotient into two qualitatively different degree ranges.  For
$$
        \lceil\tau Q\rceil\leq m<cQ,
$$
the Frobenius ideal has no elements of degree $m$, so the entire graded piece is the space $H^0(X,mL-QD)$.  These dimensions are governed asymptotically by divisor volumes.  In degrees $m\geq cQ$, the Frobenius generators are present, and the quotient measures their failure to generate all sections of $mL-QD$.  Section~\ref{sec:enriques-gHK} computes this second contribution from the actual multiplication map on global sections.

\subsection{Section Rings and Divisorial Ideals}\label{subsec:section-rings-divisorial-ideals}

Vanishing along a Cartier divisor defines a homogeneous ideal in its section ring.  Ordinary products of sections and the full space of sections of the corresponding multiple divisor define the same sheaf, although their graded modules may differ in finitely many degrees.

Let $X$ be a normal projective variety, and let $L$ be an ample line bundle such that the section ring
$$
        S:=R(X,L)=\bigoplus_{m\ge 0}H^0(X,mL)
$$
is standard graded.  Then $L$ is globally generated and
$$
        X\cong \Proj S.
$$
Indeed, since $S=R(X,L)$ is generated in degree one, the linear system $|L|$ has no base points: otherwise every section of every $mL$ would vanish at the same point, contradicting the fact that a sufficiently large multiple of an ample line bundle is globally generated.  Hence $L$ is globally generated.  The canonical morphism
$X \longrightarrow \Proj R(X,L)$
associated to an ample invertible sheaf is an open immersion with dense image; see \cite[Tag~01Q1]{Stacks}.  Since $X$ is projective over $k$, it is proper over $k$, while $\Proj S$ is separated over $k$.  Hence this canonical morphism is proper.  A proper open immersion has closed image, and its image is already dense.  It is therefore surjective and hence an isomorphism.  Thus $X\cong\Proj S$.  Under this isomorphism,
$\OO_{\Proj S}(1)\cong L$.
Indeed, the canonical morphism carries the natural map from the pullback of $\OO_{\Proj S}(1)$ to $L$; on every standard open defined by a nonvanishing section of degree one, the two line bundles are trivialized by that same section, so this map is an isomorphism.

For a coherent sheaf $\F$ on $X$, put
$$
        \Gamma_*(\F):=\bigoplus_{m\ge 0}H^0(X,\F\otimes L^{\otimes m}).
$$
Let $D$ be an effective Cartier divisor on $X$.  We define the corresponding divisorial ideal in the section ring by
$$
I_D:=\Gamma_*(\OO_X(-D))
=
\bigoplus_{m\ge 0}H^0(X,mL-D)
\subseteq R(X,L).
$$
More generally,
$$
I_{nD}:=\Gamma_*(\OO_X(-nD))
=
\bigoplus_{m\ge 0}H^0(X,mL-nD).
$$
Since $D$ is Cartier, $\OO_X(-D)$ is invertible, and hence
$\widetilde{I_D}\cong \OO_X(-D), \widetilde{I_{nD}}\cong \OO_X(-nD)$.
Moreover, sheafification commutes with tensor products of invertible sheaves on standard opens, so
$\widetilde{I_D^n}\cong \OO_X(-nD)$.
Thus the product $I_D^n$ and the divisorial module $I_{nD}$ define the same coherent ideal sheaf on $X$.

The difference between these two graded ideals is a saturation phenomenon.  The form used below requires that $S$ be the full section ring.

\begin{lemma}\label{lem:full-section-ring-saturation}
Let $S=R(X,L)$ be a standard graded full section ring as above, and identify $X$ with $\Proj S$.  If $I\subseteq S$ is a homogeneous ideal, then, inside
$S_m=H^0(X,mL)$,
one has
$$
        (I:S_+^\infty)_m
        =H^0\bigl(X,\widetilde I\otimes mL\bigr)
        \qquad(m\geq0).
$$
Equivalently,
$$
        I^{\rm sat}=\Gamma_*(\widetilde I)
$$
as homogeneous submodules of the full section ring $S$.
\end{lemma}

\begin{proof}
Let $x_1,\ldots,x_r$ be generators of $S$ in degree one, so that the standard opens $D_+(x_i)$ cover $X$.
If $s\in(I:S_+^\infty)_m$, then $S_+^Ns\subseteq I$ for some $N$. Hence $x_i^Ns\in I$ for every $i$, and therefore $s\in(I_{x_i})_m$.  Since
$\Gamma\bigl(D_+(x_i),\widetilde I\otimes mL\bigr)
=(I_{x_i})_m,$ the section $s$ belongs to
$H^0(X,\widetilde I\otimes mL)$.
Conversely, let
$$
s\in H^0(X,\widetilde I\otimes mL)
\subseteq H^0(X,mL)=S_m.
$$
On each $D_+(x_i)$ one has $s\in(I_{x_i})_m$, so there is an integer
$n_i$ such that $x_i^{n_i}s\in I$.  If $n=\max_i n_i$, then
$$
(x_1^n,\ldots,x_r^n)s\subseteq I.
$$
Since
$$
\sqrt{(x_1^n,\ldots,x_r^n)}=S_+,
$$
some power of $S_+$ is contained in $(x_1^n,\ldots,x_r^n)$.  Hence
$S_+^Ns\subseteq I$ for some $N$, and therefore
$s\in (I:S_+^\infty)$.
\end{proof}

Before turning to the asymptotic estimates, we illustrate the two degree ranges in a case where every summand can be counted exactly. It also shows that the degrees in which the Frobenius generators act cannot, in general, be discarded.

\begin{example}\label{ex:veronese-lower-upper}
Let $X=\PP^1_k$ with homogeneous coordinates $[u:v]$, let
$$
        L=\OO_{\PP^1}(2),
        \qquad
        D=V(u),
$$
and let
$$
        S=R(X,L)=k[u^2,uv,v^2]
$$
with its standard grading.  Since $L-D=\OO_{\PP^1}(1)$, multiplication by the canonical section $u$ of $D$ embeds
$$
        H^0(\PP^1,\OO_{\PP^1}(1))=\langle u,v\rangle_k
$$
into $S_1$, and the resulting ideal is
$$
        J=(u^2,uv)\subseteq S.
$$
Assume that $p>2$, and let $Q=p^e=2r+1$.

For $m<Q$, the Frobenius ideal has no component of degree $m$, while Lemma~\ref{lem:frobenius-saturation-by-degree} gives
$$
        \bigl((J^{[Q]})^{\rm sat}\bigr)_m
        =H^0\bigl(\PP^1,\OO_{\PP^1}(2m-Q)\bigr).
$$
Consequently the total contribution from the degrees below $Q$ is
$$
\begin{aligned}
        \sum_{0\leq m<Q}
        \dim_k\left(
        \frac{(J^{[Q]})^{\rm sat}}{J^{[Q]}}
        \right)_m
        & =\sum_{m=r+1}^{2r}(2m-2r) \\
        & =r(r+1).
\end{aligned}
$$

Now write $m=Q+a$ with $a\geq0$.  By Lemma~\ref{lem:frobenius-saturation-by-degree}, multiplication by $u^Q$ identifies the degree $Q+a$ part of the saturation with
$$
        \bigl((J^{[Q]})^{\rm sat}\bigr)_{Q+a}
        =u^QH^0\bigl(\PP^1,\OO_{\PP^1}(Q+2a)\bigr)
        \subseteq S_{Q+a}.
$$
On the other hand,
$$(J^{[Q]})_{Q+a}=(u^{2Q},u^Qv^Q)S_a=u^Q\langle u^Q,v^Q\rangle_k H^0\bigl(\PP^1,\OO_{\PP^1}(2a)\bigr).$$
Since multiplication by $u^Q$ is injective, cancelling this common factor identifies the degree $Q+a$ part of the saturation quotient with the cokernel of
$$
        \langle u^Q,v^Q\rangle_k
        \otimes H^0\bigl(\PP^1,\OO_{\PP^1}(2a)\bigr)
        \longrightarrow
        H^0\bigl(\PP^1,\OO_{\PP^1}(Q+2a)\bigr).
$$
Equivalently, this is the map
$$
        H^0\bigl(\PP^1,\OO_{\PP^1}(1)\bigr)^{(Q)}
        \otimes H^0\bigl(\PP^1,\OO_{\PP^1}(2a)\bigr)
        \longrightarrow
        H^0\bigl(\PP^1,\OO_{\PP^1}(Q+2a)\bigr)
$$
induced by taking $Q$th powers in the first factor.  The cokernels $C_{Q,a}$ introduced in Subsection~\ref{subsec:split-pencil-bundle} are the higher dimensional analogues of this map.
The target has the monomial basis
$$
        u^iv^{Q+2a-i},
        \qquad 0\leq i\leq Q+2a.
$$
The image contains exactly the monomials with $i\leq2a$ or $i\geq Q$.  Hence the missing monomials are those with
$$
        2a<i<Q,
$$
and the cokernel has dimension
$$
        \max\{Q-2a-1,0\}.
$$
The total contribution from the degrees at least $Q$ is therefore
$$
        \sum_{a\geq0}\max\{Q-2a-1,0\}
        =r(r+1).
$$
Thus
$$
        \lambda_S\left(
        \frac{(J^{[Q]})^{\rm sat}}{J^{[Q]}}
        \right)
        =2r(r+1)
        =\frac{Q^2-1}{2},
$$
and
$$
        e_{\rm gHK}(JS_{S_+})=\frac{1}{2}.
$$
Here the lower and upper ranges contribute equally, so the upper multiplication cokernel is not a correction of lower order.  In the Enriques construction of Section~\ref{sec:enriques-gHK}, the upper contribution requires a separate calculation of its logarithmic part and sign.\hfill\qedsymbol
\end{example}

\subsection{Linear Frobenius--Serre Vanishing}\label{subsec:frobenius-serre-vanishing}

The next estimate is used in two places.  In
Lemma~\ref{lem:upper-sum-finite}, it bounds the degrees in which the upper cokernels can be nonzero.  In
Lemma~\ref{lem:fixed-homogeneous-LC}, it gives the linear LC bound for Frobenius powers of homogeneous ideals.  Both applications require vanishing once the twist is larger than a fixed multiple of $Q$.
The estimate follows from the finite regularity resolution of Arapura~\cite[Corollary~3.2]{ArapuraFrobenius}, Frobenius flatness, and
Serre vanishing.

\begin{lemma}\label{lem:linear-frobenius-serre-vanishing}
Let $X$ be a regular projective variety over a field of characteristic $p>0$, and let $L$ be an ample line bundle on $X$.  For every coherent sheaf $\F$ on $X$, there is a constant $C_\F>0$ such that, for every
$Q=p^e$ and every integer $m\geq C_\F Q$, one has
$$
H^i\bigl(X,F^{e*}\F\otimes mL\bigr)=0
\qquad\text{for all }i>0.
$$
\end{lemma}

\begin{proof}
Choose $s\geq1$ such that $A:=sL$ is very ample, and put
$d=\dim X$.  Choose $m_0$ such that $\F$ is $m_0$-regular with
respect to $A$, and set
$$
R_A=\max\{1,\operatorname{reg}_A(\OO_X)\}.
$$
By \cite[Corollary~3.2]{ArapuraFrobenius}, truncated after $d$ steps,
there is an exact sequence
$$
E_d\longrightarrow E_{d-1}\longrightarrow\cdots
\longrightarrow E_0\longrightarrow\F\longrightarrow0,
$$
where
$$
E_j=V_j\otimes A^{-(m_0+jR_A)}
$$
for finite-dimensional vector spaces $V_j$.  Since $X$ is regular, the
absolute Frobenius is flat by Kunz's theorem \cite{Kunz69}; pulling back
this sequence by $F^e$ therefore preserves exactness.

Write $m=sq+r$ with $0\leq r<s$.  After Frobenius pullback and twisting
by $mL=qA+rL$, every summand in the $j$th term has the form
$$
A^{q-Q(m_0+jR_A)}\otimes rL.
$$
There are only finitely many pairs $(j,r)$.  Serre vanishing therefore
gives an integer $M\geq0$ such that
$$
H^i\bigl(X,A^t\otimes rL\bigr)=0
\qquad
(i>0,\ t\geq M,\ 0\leq r<s).
$$
Set
$$
B=\max_{0\leq j\leq d}\{0,m_0+jR_A\}.
$$
If $q\geq BQ+M$, then $q-Q(m_0+jR_A)\geq M$ for every $j$, so every
term in the pulled-back and twisted sequence has vanishing higher
cohomology.

Set
$$
\mathcal E_j:=F^{e*}E_j\otimes mL,
\qquad
\mathcal G:=F^{e*}\F\otimes mL,
$$
and put $\mathcal K_0:=\mathcal G$.  For
$0\leq j\leq d-1$, define
$$
\mathcal K_{j+1}
:=
\ker\bigl(\mathcal E_j\longrightarrow\mathcal K_j\bigr),
$$
where the map is induced by the pulled-back and twisted resolution. Exactness gives short exact sequences
$$
0\longrightarrow\mathcal K_{j+1}
\longrightarrow\mathcal E_j
\longrightarrow\mathcal K_j
\longrightarrow0
\qquad
(0\leq j\leq d-1).
$$
Each $\mathcal K_j$ is coherent.  Since
$$
H^t(X,\mathcal E_j)=0
\qquad(t>0),
$$
the associated long exact cohomology sequences give
$$
H^t(X,\mathcal K_j)
\cong
H^{t+1}(X,\mathcal K_{j+1})
\qquad(t>0).
$$
Iterating, for every $i>0$ one obtains
$$
H^i(X,\mathcal G)
\cong
H^{i+d}(X,\mathcal K_d).
$$
Since $i+d>d=\dim X$, Grothendieck vanishing gives
$$
H^{i+d}(X,\mathcal K_d)=0.
$$
Therefore
$$
H^i\bigl(X,F^{e*}\F\otimes mL\bigr)=0
\qquad(i>0)
$$
whenever $q\geq BQ+M$. Finally, set
$$
C_\F:=s(B+M+1).
$$
If $m\geq C_\F Q$, then
$$
q =
\left\lfloor\frac{m}{s}\right\rfloor
\geq
(B+M+1)Q-1
\geq
BQ+M.
$$
This proves the lemma.
\end{proof}

\subsection{Volumes, Pseudo-effective Thresholds, and Uniform Asymptotics}\label{subsec:volume-prelim}

The degreewise section counts arising from saturation will eventually be summed over an interval whose length grows with the Frobenius exponent.  Pointwise asymptotics are therefore insufficient; we need one error estimate that is uniform as the numerical class varies through a compact segment.

Let $X$ be a projective variety of dimension $r$.  For a Cartier divisor $A$ on $X$, its volume is
$$
\vol_X(A)
:=
\limsup_{m\to\infty}
\frac{h^0(X,mA)}{m^r/r!}.
$$
The volume measures the asymptotic growth of spaces of global sections of the line bundles determined by Cartier divisors.  It depends only on the numerical class of $A$ and extends continuously to $N^1(X)_\RR$; see Lazarsfeld \cite[Theorem~2.2.44]{Lazarsfeld}.  It is homogeneous of degree $r$:
$$
\vol_X(a\alpha)=a^r\vol_X(\alpha)
\qquad
(a\in\RR_{\ge 0},\ \alpha\in N^1(X)_\RR),
$$
and if $A$ is nef, then
$\vol_X(A)=A^r$.
The continuity and convex geometry of volume functions are studied in the work of Lazarsfeld-Musta\c{t}\u{a} \cite{LM} and K\"uronya-Lozovanu-Maclean \cite{KLM}; here we need only continuity and a uniform asymptotic estimate for section counts on compact subsets of $N^1(X)_\RR$.

Now let $L$ be ample and let $D$ be an effective Cartier divisor on $X$.  We define the pseudo-effective threshold of $D$ with respect to $L$ by
$\tau=\tau_L(D) := \inf\{\,t\ge 0\mid tL-D\text{ is pseudo-effective}\,\}$.
This is the first point at which the numerical ray $tL-D$ meets the pseudo-effective cone.  In the applications below, $cL-D$ is globally generated, hence effective, and therefore pseudo-effective.  Thus $\tau\le c$ in the range relevant to the one ideal formula.

We use the specialization below of the uniform asymptotic estimate for finite families of real Cartier divisors in \cite[Proposition~3.5.1]{BGGJKM}.  It is stated in the form required by the Frobenius Riemann sums below.

\begin{lemma}\label{lem:uniform-volume-asymptotics}
Let $X$ be a projective variety of dimension $r$, and let $L$ and $D$ be Cartier divisors on $X$.  Fix real numbers
$0\leq a<b$.
Then there exists a function $\rho(n)=o(n^r)$ such that, for all sufficiently large integers $n$ and every integer $m$ with
$\lceil an\rceil\leq m\leq \lfloor bn\rfloor$,
one has
$$
\left|
h^0(X,mL-nD)
-
\frac{n^r}{r!}\vol_X\left(\frac{m}{n}L-D\right)
\right|
\leq \rho(n).
$$
\end{lemma}

\begin{proof}
Apply \cite[Proposition~3.5.1]{BGGJKM} to the fixed family $L,-D$.  It gives a function $\rho_0(M)=o(M^r)$ such that, for all $m,n\geq1$,
$$
\left|
h^0(X,mL-nD)-\frac{1}{r!}\vol_X(mL-nD)
\right|
\leq \rho_0(m+n).
$$
If $a=0$ and $m=0$, the same conclusion follows by applying the cited proposition to the family consisting only of $-D$; after enlarging $\rho_0$, we may include this case as well.

Set
$\varepsilon(n):= \sup_{M\geq n}\frac{|\rho_0(M)|}{M^r}$.
Then $\varepsilon(n)\to0$.  On the stated range,
$n\leq m+n\leq (b+1)n+1$,
and hence
$$
|\rho_0(m+n)|
\leq
\varepsilon(n)((b+1)n+1)^r
=: \rho(n),
$$
where $\rho(n)=o(n^r)$.  Finally,
$\vol_X(mL-nD) = n^r\vol_X\left(\frac{m}{n}L-D\right)$
by homogeneity of volume, and the claimed estimate follows.
\end{proof}

Let $0\leq\tau<c$.  The interval
$\lceil\tau n\rceil\leq m<cn$
contains $O(n)$ integers, so Lemma~\ref{lem:uniform-volume-asymptotics} shows that replacing each $h^0(X,mL-nD)$ by
$\frac{n^r}{r!}\vol_X\left(\frac{m}{n}L-D\right)$
produces a total error $o(n^{r+1})$.  After division by $n^{r+1}$, the remaining sum is the Riemann sum of the continuous function $t\mapsto\vol_X(tL-D)$.  Therefore
$$
\lim_{n\to\infty}
\frac{1}{n^{r+1}}
\sum_{\lceil\tau n\rceil\leq m<cn}
h^0(X,mL-nD)
=
\frac{1}{r!}
\int_{\tau}^{c}\vol_X(tL-D)\,dt.
$$
Changing either endpoint convention affects only $O(1)$ summands.  On the compact interval $[\tau,c]$, each such summand is $O(n^r)$ by Lemma~\ref{lem:uniform-volume-asymptotics} and boundedness of the volume function, so the normalized endpoint contribution tends to zero.

\begin{remark}\label{rem:why-uniformity-needed}
Uniformity is needed because the normalized class $\frac{m}{n}L-D$ varies over a compact segment containing $O(n)$ lattice points, and a pointwise asymptotic estimate would not by itself control the sum.
\end{remark}

Lemma~\ref{lem:uniform-volume-asymptotics} converts the lower degree sum into a volume integral.  Lemma~\ref{lem:weighted-pb-volume} computes that volume on the split projective bundle used below.

\subsection{Projective Bundle Convention and Riemann Sum Volumes}\label{subsec:projective-bundle-prelim}

The monomial decomposition of symmetric powers of a split bundle of rank three gives the triangle used in the main calculation.  We fix the projective bundle convention before passing to the volume formula.

We use the quotient convention for projective bundles:
$$
        \mathbf P_Y(E):=\Proj_Y \Sym^\bullet E.
$$
Thus $\mathbf P_Y(E)$ parametrizes one dimensional quotients of the fibers of $E$, the tautological line bundle satisfies
$$
        \pi^*E\twoheadrightarrow \OO_{\mathbf P_Y(E)}(1),
$$
and
$$
        \pi_*\OO_{\mathbf P_Y(E)}(m)=\Sym^m E
        \qquad(m\geq 0).
$$
See \cite[Tag~01OA]{Stacks}.  This convention is used throughout the projective bundle constructions below.

The volume computations use a standard Riemann sum mechanism.  If
$$
        E=\bigoplus_{i=0}^r\OO_Y(A_i),
$$
then sections of $m\OO_{\mathbf P_Y(E)}(1)$ decompose into summands indexed by lattice points in a simplex:
$$
\Sym^m E
=
\bigoplus_{n_0+\cdots+n_r=m}
\OO_Y(n_0A_0+\cdots+n_rA_r).
$$
After dividing by $m$, these lattice points become a Riemann sum over the simplex.  Combining this with the uniform volume estimate from Subsection~\ref{subsec:volume-prelim} gives the projective bundle volume formula proved in Lemma~\ref{lem:weighted-pb-volume}.

\subsection{A Logarithmic Transcendence Input}\label{subsec:baker-prelim}

The geometric integrals below evaluate to an algebraic number plus logarithms of positive algebraic numbers.  Baker's theorem turns nonvanishing of the logarithmic part into transcendence; we use the form below from Waldschmidt \cite[Theorem~11.1]{Waldschmidt}.  Throughout, a logarithm of a positive algebraic number means its real logarithm.

We say that a real number $\Theta$ has a \emph{presentation by algebraic numbers and logarithms} if
$$
        \Theta=A+\Lambda,
        \qquad
        A\in\overline{\QQ},
        \qquad
        \Lambda=\sum_{j=1}^r b_j\log\alpha_j,
$$
where $b_j\in\overline{\QQ}$ and $\alpha_j\in\overline{\QQ}_{>0}$.  Once Theorem~\ref{thm:Baker} is available, the numerical value of $\Lambda$ is independent of the chosen presentation: the difference of two such logarithmic sums is algebraic, and Baker's theorem forces that difference to be zero.  We therefore call $\Lambda$ the \emph{logarithmic part} of $\Theta$.

\begin{theorem}\label{thm:Baker}
Let $\alpha_1,\ldots,\alpha_m$ be positive algebraic numbers whose real logarithms
$$
        \log\alpha_1,\ldots,\log\alpha_m
$$
are linearly independent over $\QQ$.  Then
$$
        1,\log\alpha_1,\ldots,\log\alpha_m
$$
are linearly independent over $\overline{\QQ}$.
\end{theorem}

We use the resulting consequence.

\begin{corollary}\label{cor:Baker-linear-form}
Let $\ell_1,\dots,\ell_n$ be real logarithms of positive algebraic numbers.  Let $A_0,\dots,A_n\in\overline{\QQ}$, and put
$$
\Lambda
=
A_0+\sum_{j=1}^n A_j\ell_j.
$$
If the logarithmic part
$$
        \sum_{j=1}^n A_j\ell_j
$$
is nonzero, then $\Lambda$ is transcendental.
\end{corollary}

\begin{proof}
Choose a maximal subfamily $u_1,\dots,u_s$ of the logarithms that is linearly independent over $\QQ$ and rewrite the logarithmic part as $\sum_i B_i u_i$, with $B_i\in\overline{\QQ}$ and not all $B_i$ zero.  If $\Lambda$ were algebraic, then
$(A_0-\Lambda)+\sum_{i=1}^s B_i u_i=0$
would contradict Theorem~\ref{thm:Baker}.
\end{proof}

In the application, the logarithms need not be independent as first written.  Passing to a maximal subfamily that is linearly independent over $\QQ$ reduces the problem to showing that the resulting logarithmic sum is nonzero.

\section{Volumes on Split Projective Bundles}\label{sec:geometric-input}

For a split vector bundle $E=\bigoplus_i\OO_Y(A_i)$, high powers of the tautological line bundle on $\PP_Y(E)$ decompose into summands indexed by lattice points of a simplex.  The normalized section count is therefore a Riemann sum for the volume function on $Y$.  The lemma below gives the normalization used in Section~\ref{sec:enriques-gHK}; when the rank is three, the simplex is a triangle.

\begin{lemma}\label{lem:weighted-pb-volume}
Let $Y$ be an integral projective variety of dimension $v$, let
$$
        E=\bigoplus_{i=0}^r \OO_Y(A_i)
$$
for Cartier divisors $A_0,\ldots,A_r$ on $Y$, and set
$$
X:=\mathbf P_Y(E),
\qquad
\xi:=\OO_X(1),
\qquad
\pi:X\to Y.
$$
We use the quotient convention for projective bundles, so that
$$
        \pi_*\OO_X(m\xi)=\Sym^m E
        \qquad(m\geq 0).
$$
Let $d=v+r$.  For every integer $a\geq 1$ and every Cartier divisor $B$ on $Y$, one has
$$
\vol_X(a\xi+\pi^*B)
=
\frac{d!}{v!}
\int_{\Delta_a}
\vol_Y\left(B+\sum_{i=0}^r\mu_iA_i\right)
\,d\mu_1\cdots d\mu_r,
$$
where
$$
\Delta_a
:=
\left\{
(\mu_0,\ldots,\mu_r)\in\RR_{\geq 0}^{r+1}
\ \middle|\ 
\sum_{i=0}^r\mu_i=a
\right\}.
$$
Here the integral is taken in the coordinates $\mu_1,\ldots,\mu_r$, with
$\mu_0=a-\sum_{i=1}^r\mu_i$.
\end{lemma}

\begin{proof}
For every integer $k\geq1$, the projective bundle formula and the splitting of $E$ give
$$
\begin{aligned}
h^0\left(X,k(a\xi+\pi^*B)\right)
&=h^0\left(Y,\Sym^{ka}(E)\otimes\OO_Y(kB)\right)\\
&=\sum_{n_0+\cdots+n_r=ka}
h^0\left(Y,kB+\sum_{i=0}^r n_iA_i\right).
\end{aligned}
$$
Set
$\Delta_a(k)=\Delta_a\cap \left(\frac{1}{k}\ZZ_{\geq0}\right)^{r+1}$ and $\mu_i=\frac{n_i}{k}$.
After division by $k^d/d!$, the preceding identity becomes
$$
\frac{h^0\left(X,k(a\xi+\pi^*B)\right)}{k^d/d!}
=
\frac{d!}{v!}\frac{1}{k^r}
\sum_{\mu\in\Delta_a(k)}
\frac{h^0\left(Y,k\left(B+\sum_i\mu_iA_i\right)\right)}{k^v/v!}.
$$

Apply \cite[Proposition~3.5.1]{BGGJKM} to the fixed family
$B,A_0,\ldots,A_r$.  If a coefficient $n_i$ is zero, omit the corresponding $A_i$; there are only finitely many possible supports.  For each support the cited proposition gives an error function $\rho_S(M)=o(M^v)$.  Taking the maximum over the finitely many supports gives one function
$\rho(M)=o(M^v)$
that works for every lattice point.  Since
$k+\sum_i n_i=(a+1)k$,
we have, uniformly for $\mu\in\Delta_a(k)$,
$$
\left|
\frac{h^0\left(Y,k\left(B+\sum_i\mu_iA_i\right)\right)}{k^v/v!}
-
\vol_Y\left(B+\sum_i\mu_iA_i\right)
\right|
\leq
\frac{v!\rho((a+1)k)}{k^v}
=o(1).
$$
Since $\#\Delta_a(k)=\binom{ka+r}{r}=O(k^r)$, the total unnormalized error is
$O(k^r)\,o(k^v)=o(k^{v+r})=o(k^d)$.
Finally, identify $\Delta_a$ with
$\{(\mu_1,\ldots,\mu_r)\in\RR_{\geq0}^r:\ \mu_1+\cdots+\mu_r\leq a\}$.
The volume function is continuous on the compact image of this simplex, so the remaining lattice sum is the coordinate Riemann sum for
$$
        \int_{\Delta_a}
        \vol_Y\left(B+\sum_{i=0}^r\mu_iA_i\right)
        \,d\mu_1\cdots d\mu_r.
$$
Taking $k\to\infty$ proves the formula.
\end{proof}

\begin{example}
Let $C$ be a smooth projective curve, and let $A$ and $B$ be divisors on $C$ with $\deg A=b>0$ and $\deg B=c$.  Set $E=\OO_C\oplus\OO_C(A)$, let $X=\PP_C(E)$, and write $\xi=\OO_X(1)$.

For $a\geq1$, Lemma~\ref{lem:weighted-pb-volume} gives
$$
\vol_X(a\xi+\pi^*B)
=
2\int_0^a \vol_C(B+\mu A)\,d\mu.
$$
On a curve, $\vol_C(B+\mu A)=\max\{c+\mu b,0\}$.  Hence
$$
\vol_X(a\xi+\pi^*B)
=
2\int_0^a\max\{c+\mu b,0\}\,d\mu
=
\begin{cases}
a^2b+2ac, & c\geq0,\\[4pt]
\dfrac{(ab+c)^2}{b}, & -ab<c<0,\\[8pt]
0, & c\leq-ab.
\end{cases}
$$
When $c\geq0$, the divisor is nef and the first expression is its
self intersection.  When $-ab<c<0$, the divisor is big but not nef, and only the part of the interval where $c+\mu b>0$ contributes.\hfill\qedsymbol
\end{example}

The same formula will be used below for real numerical divisor classes.  First let
$a\in\QQ_{>0}$, $B\in N^1(Y)_\QQ$.
Choose $q\geq1$ such that $qa\in\ZZ_{\geq1}$ and $qB$ is represented by a Cartier divisor.  Apply the integral formula to
$q(a\xi+\pi^*B)=qa\,\xi+\pi^*(qB)$.
Homogeneity of volume and the substitution $\nu_i=q\mu_i$ on $\Delta_{qa}$ contribute the same factor
$q^d=q^{v+r}$
on the two sides, and therefore give the formula for rational $(a,B)$.

For continuity, write the right side on the fixed simplex $\Delta_1$ as
$$
\frac{d!}{v!}
        a^r
        \int_{\Delta_1}
        \vol_Y\left(B+a\sum_{i=0}^r\lambda_iA_i\right)
        \,d\lambda_1\cdots d\lambda_r.
$$
The volume function is uniformly continuous on compact subsets of $N^1(Y)_\RR$, so this expression is continuous in $(a,B)$ for $a>0$.  The left side is also continuous in the numerical class.  Since rational pairs are dense, the formula extends to every real $a>0$ and every $B\in N^1(Y)_\RR$.  In particular, it applies to the real classes $\xi-u\pi^*S_Y$ and $tL_N-D_N$.

\section{A Split Pencil Enriques Construction for Generalized Hilbert--Kunz Multiplicity}\label{sec:enriques-gHK}

We prove Theorem~\ref{thm:intro-ghk} using three genus one pencils on an unnodal Enriques surface.  They determine a split projective bundle and a six-dimensional linear system.  The resulting Frobenius saturation defect has two ranges.  Below the Frobenius generating degree, it is governed by divisor volumes; at and above that degree, it is the cokernel of a multiplication map.  We compute the logarithmic contribution of the second range and show that it has the same sign as the first.

The choice of surface is dictated by the calculation.  The features used below and their roles are as follows.

\begin{center}
\renewcommand{\arraystretch}{1.12}
\begin{tabular}{|p{0.34\textwidth}|p{0.54\textwidth}|}
\hline
\textbf{Geometric feature} & \textbf{Role in the argument} \\
\hline
Three base point free genus one pencils & Supply the six generators of degree one and the monomial decomposition for a bundle of rank three after Frobenius. \\
\hline
$F_i^2=0$ and $F_iF_j=\eta$ for $i\neq j$ & Make the intersection form symmetric and reduce the volume integrand to a radial quadratic expression. \\
\hline
Unnodality & Removes negative curves, so the nef and pseudo-effective cones coincide with the selected positive cone. \\
\hline
The split bundle of rank three & Its monomial summands are indexed by triples of nonnegative integers; after normalization, these indices fill a two-dimensional simplex. \\
\hline
The boundary of the positive cone & Cuts the centered simplex by a conic and produces the logarithmic term in the volume integral. \\
\hline
\end{tabular}
\end{center}

Suppose that $F_i^2=0$, that $F_iF_j=\eta$ for $i\neq j$, and that $c_i=r+x_i$ with $x_0+x_1+x_2=0$.  Then
$$
        \left(\sum_{i=0}^2c_iF_i\right)^2
        =\eta\bigl(6r^2-R^2\bigr),
        \qquad
        R^2=x_0^2+x_1^2+x_2^2.
$$
After we center the coefficient simplex, the positive cone meets it in a circle.  This gives the radial integral below. An unnodal Enriques surface supplies the required pencils and has no negative curves, so the volume on their span is governed by this quadratic form. Three pencils are needed for the present construction. With only two, the exponent simplex is an interval, and integrating the positive part of the resulting quadratic gives an algebraic value. The
third pencil makes the parameter space two-dimensional; the boundary of the positive cone is then a conic, and the resulting integral can contain a logarithm.

Throughout this section, $k$ is an uncountable algebraically closed field of characteristic $p>2$, and $Q=p^e$.

\subsection{The Enriques Pencil Package}\label{subsec:split-pencil-input}

The construction uses three nef isotropic classes with a symmetric intersection matrix.  Unnodality identifies both the nef and pseudo-effective cones with the closed positive cone; on the span of these classes, the volume function is therefore determined entirely by the intersection form.

\begin{lemma}\label{lem:enriques-pencil-package}
There exists a smooth classical unnodal Enriques surface $Y$ over $k$ and divisor classes
$F_0,F_1,F_2\in \operatorname{Pic}(Y)$
such that each $|F_i|$ is a base point free genus one pencil and
$$
        F_i^2=0,\qquad F_i\cdot F_j=4\quad (i\neq j).
$$
If
$S_Y:=F_0+F_1+F_2$,
then $S_Y$ is ample. Define the closed positive cone determined by $S_Y$ as follows:
$$
        \mathcal C^+
        :=\{\alpha\in N^1(Y)_\RR\mid \alpha^2\geq0,
        \ \alpha\cdot S_Y\geq0\}.
$$
Then
$$
        \operatorname{Nef}(Y)
        =\overline{\operatorname{Eff}}(Y)
        =\mathcal C^+.
$$
Consequently, for every $\alpha\in N^1(Y)_\RR$,
$$
        \vol_Y(\alpha)=
        \begin{cases}
        \alpha^2, & \alpha\in\mathcal C^+,\\
        0, & \alpha\notin\mathcal C^+.
        \end{cases}
$$
In particular, this description holds on the subspace
$V=\spann_\RR\{F_0,F_1,F_2\}\subseteq N^1(Y)_\RR$,
which is the only part of $N^1(Y)_\RR$ used below.
\end{lemma}

\begin{proof}
By Martin’s result \cite[Theorem~B]{MartinUnnodal}, the generic Enriques surface over an algebraically closed field of odd characteristic is unnodal. In particular, unnodal Enriques surfaces exist over k; fix one and denote it by Y.

Every Enriques surface admits a half-fiber, equivalently a $1$-sequence; see \cite[Introduction]{MartinMezzedimiVeniani}.  Choose one and denote it by $E_0$.  In particular, the class of $E_0$ is primitive in $\operatorname{Num}(Y)$.

By Martin--Mezzedimi--Veniani~\cite[Theorem~2.5 and the discussion following it]{MartinMezzedimiVeniani}, the sequence $(E_0)$ extends to a $c'$ degenerate $10$-sequence for some $c'\geq1$.  By definition, the entries beyond its $c'$ half-fibers are obtained by adding chains of $(-2)$-curves.  Since $Y$ is unnodal, no such chains occur.  Hence $c'=10$, and the extension is an ordinary $10$-sequence of half-fibers
$E_0,\ldots,E_9$
satisfying
$$
        E_i^2=0,
        \qquad
        E_i\cdot E_j=1
        \quad(i\neq j).
$$
Choose the first three and put
$F_i:=2E_i$ $(0\leq i\leq2)$.
Let $f_i:Y\to\PP^1$ be the genus one fibration with half-fiber $E_i$.  Since $2E_i$ is a fiber of $f_i$,
$\OO_Y(F_i)=\OO_Y(2E_i)\cong f_i^*\OO_{\PP^1}(1)$.
The fibers of $f_i$ are connected, so $(f_i)_*\OO_Y=\OO_{\PP^1}$ and hence
$h^0(Y,F_i)=h^0(\PP^1,\OO_{\PP^1}(1))=2$.
Thus the complete linear system $|F_i|$ is the base point free pencil defining $f_i$, and
$F_i^2=0$ and $F_i\cdot F_j=4 (i\neq j)$.
The class $S_Y=F_0+F_1+F_2$ is nef and
$S_Y^2=2(F_0F_1+F_0F_2+F_1F_2)=24>0$.
If $C$ is an irreducible curve with $S_Y\cdot C=0$, then the Hodge index theorem gives $C^2<0$.  Since $K_Y$ is numerically trivial, adjunction gives
$C^2=2p_a(C)-2$.
Thus $p_a(C)=0$ and $C^2=-2$.  The normalization sequence shows that an integral curve with $p_a(C)=0$ is smooth and rational.  Thus $C$ is a $(-2)$-curve, contradicting the assumption that $Y$ is unnodal.  Hence $S_Y$ intersects every irreducible curve positively, and Nakai--Moishezon shows that $S_Y$ is ample.

We first show that the nef cone is $\mathcal C^+$.  Every irreducible curve $C$ on $Y$ satisfies $C^2\geq0$, because a curve of negative self-intersection would again be a $(-2)$-curve.  Since $S_Y\cdot C>0$, the numerical class of $C$ belongs to $\mathcal C^+$.  Two classes in the same closed positive cone have nonnegative intersection.  Indeed, set
$h:=\frac{S_Y}{\sqrt{S_Y^2}}$,
and write
$\alpha=ah+u, \beta=bh+v, u,v\in h^\perp$.
Since $\alpha\cdot h=a\geq0$ and $\beta\cdot h=b\geq0$, the Hodge index theorem and the conditions $\alpha^2,\beta^2\geq0$ give
$a\geq\sqrt{-u^2}, b\geq\sqrt{-v^2}$.
Cauchy--Schwarz on the negative definite space $h^\perp$ therefore yields
$$
        \alpha\cdot\beta
        =ab+u\cdot v
        \geq ab-\sqrt{(-u^2)(-v^2)}
        \geq0.
$$
It follows that every class in $\mathcal C^+$ has nonnegative intersection with every curve and is therefore nef.  For the reverse inclusion, let $0\neq\alpha\in\operatorname{Nef}(Y)$.  Then $\alpha^2\geq0$.  Moreover, $\alpha\cdot S_Y>0$; otherwise $\alpha\in S_Y^\perp$, and the Hodge index theorem would give $\alpha^2<0$.  Hence $\alpha$ belongs to the closed positive cone selected by $S_Y$, and therefore
$\operatorname{Nef}(Y)=\mathcal C^+$.
The same intersection inequality also shows that $\mathcal C^+$ is convex: if $\alpha,\beta\in\mathcal C^+$ and $r,s\geq0$, then
$$
        (r\alpha+s\beta)^2
        =r^2\alpha^2+2rs\alpha\cdot\beta+s^2\beta^2\geq0,
$$
and $(r\alpha+s\beta)\cdot S_Y\geq0$.

We identify the pseudo-effective cone.  Since the nef cone is the closure of the ample cone and every ample class is big, hence pseudo-effective, one always has
$\operatorname{Nef}(Y) \subseteq \overline{\operatorname{Eff}}(Y)$.
Using the equality just proved, this gives
$\mathcal C^+ \subseteq \overline{\operatorname{Eff}}(Y)$.
For the reverse inclusion, every irreducible curve $C$ has already been shown to satisfy
$C^2\geq0, S_Y\cdot C>0$,
so its numerical class lies in $\mathcal C^+$.  Hence every effective divisor class, being a nonnegative linear combination of irreducible curve classes, lies in $\mathcal C^+$.  Since $\mathcal C^+$ is closed, taking closures yields
$\overline{\operatorname{Eff}}(Y) \subseteq \mathcal C^+$.
Therefore
$$
        \operatorname{Nef}(Y)
        =\overline{\operatorname{Eff}}(Y)
        =\mathcal C^+.
$$
Every class in $\mathcal C^+$ is nef, and hence its volume equals its self-intersection.  A class outside $\mathcal C^+$ is not pseudo-effective and therefore has volume zero.  This proves the global volume formula; restricting it to
$V=\spann_\RR\{F_0,F_1,F_2\}$
gives the form used in the subsequent calculations.
\end{proof}

\begin{remark}\label{rem:characteristic-range}
The argument uses only that the characteristic is different from $2$.  In odd characteristic every Enriques surface is classical, and the unnodal half-fiber package used in Lemma~\ref{lem:enriques-pencil-package} is available.  Characteristic $2$ is excluded because classical, singular, and supersingular Enriques surfaces require separate geometric input.
\end{remark}

For the rest of the section we fix $Y$ and $F_0,F_1,F_2$ as in Lemma~\ref{lem:enriques-pencil-package}, and write
$\eta=F_i\cdot F_j=4  (i\neq j)$.
We retain the symbol $\eta$ in the calculations below and call
$\RR_{\geq0}F_0+\RR_{\geq0}F_1+\RR_{\geq0}F_2$
the nef pencil cone.  Every class in this cone is nef, because each $F_i$ is nef.

\subsection{The Split Projective Bundle and the Frobenius Ideal}\label{subsec:split-pencil-bundle}

We now place the three pencils in a single projective bundle.  The splitting of the bundle makes the sections of powers of $\OO_X(1)$ decompose into summands indexed by triples of nonnegative integers, which is the form used in the Frobenius calculation.  To work in a standard graded ring, we twist $\OO_X(1)$ by a sufficiently large multiple of $\pi^*S_Y$.  The resulting line bundle has a normal full section ring generated in degree one, and the six pencil sections will define the ideal $J_N$ in this ring.

Set
$$
        E=\OO_Y(F_0)\oplus\OO_Y(F_1)\oplus\OO_Y(F_2),
        \qquad
        X=\PP_Y(E),
        \qquad
        \xi=\OO_X(1),
$$
using the quotient convention.  Thus
$\pi_*\OO_X(m\xi)=\Sym^mE (m\geq0)$.
For $N\geq1$, put
$$
        L_N:=\xi+N\pi^*S_Y.
$$
Here $\xi$ is ample on the fibers of $\pi$, while $\pi^*S_Y$ supplies positivity from the base.  Thus $L_N=\xi+N\pi^*S_Y$ is ample for $N\gg0$.  Lemma~\ref{lem:LN-normal-standard} also shows that the full section ring $R(X,L_N)$ is normal and generated in degree one.

\begin{lemma}\label{lem:LN-normal-standard}
For all sufficiently large integers $N$, the line bundle $L_N$ is ample and the full section ring
$S_N:=R(X,L_N)$
is a normal domain generated in degree one.
\end{lemma}

\begin{proof}
The bundle $\xi$ is $\pi$-ample and $S_Y$ is ample on $Y$.  Hence
$L_N=\xi+N\pi^*S_Y$ is ample for $N\gg0$; see \cite[Tag~0892]{Stacks}.  Put
$A_{i,N}=NS_Y+F_i (0\leq i\leq2)$.
The projective bundle formula gives
$$
H^0(X,mL_N)=
\bigoplus_{\beta_0+\beta_1+\beta_2=m}
H^0\left(Y,\sum_j\beta_jA_{j,N}\right).
$$
Thus generation in degree one follows once, for every nonzero
$\beta\in\ZZ_{\geq0}^3$ and every $i$, the multiplication map
$$
H^0\left(Y,\sum_j\beta_jA_{j,N}\right)\otimes H^0(Y,A_{i,N})
\longrightarrow
H^0\left(Y,\sum_j\beta_jA_{j,N}+A_{i,N}\right)
$$
is surjective.

Let $p_1,p_2:Y\times Y\to Y$ be the two projections, let
$\Delta_Y\subseteq Y\times Y$ be the diagonal, and set
$H=p_1^*S_Y+p_2^*S_Y$.
Since $S_Y$ is ample, $H$ is ample on $Y\times Y$.  If
$m=|\beta|\geq1$, then
$$
p_1^*\left(\sum_j\beta_jA_{j,N}\right)+p_2^*A_{i,N}
=NH+P_{\beta,i,N},
$$
where
$$
P_{\beta,i,N}
=p_1^*\left((m-1)NS_Y+\sum_j\beta_jF_j\right)+p_2^*F_i.
$$
The divisor $P_{\beta,i,N}$ is nef. Indeed each $F_j$ is nef, $S_Y$ is
ample and hence nef, all coefficients in the displayed expression are
nonnegative, and pullbacks and sums preserve nefness.  Keeler's uniform
Fujita vanishing \cite[Theorem~1.5]{Keeler}, in the form verified by the
corrigendum \cite{KeelerCorrigendum}, gives a bound depending only on
$\Ical_{\Delta_Y}$ and the ample bundle $H$, and uniform over all nef
twists.  It therefore applies uniformly to the divisors
$P_{\beta,i,N}$ and yields
$$
H^1\left(Y\times Y,\Ical_{\Delta_Y}\otimes
p_1^*\left(\sum_j\beta_jA_{j,N}\right)\otimes p_2^*A_{i,N}\right)=0
$$
for all $N\gg0$, uniformly in $\beta$ and $i$.

Set
$$
\mathcal L_{\beta,i,N}
:=p_1^*\OO_Y\left(\sum_j\beta_jA_{j,N}\right)
\otimes p_2^*\OO_Y(A_{i,N}).
$$
The diagonal exact sequence
$$
0\longrightarrow
\Ical_{\Delta_Y}\otimes\mathcal L_{\beta,i,N}
\longrightarrow
\mathcal L_{\beta,i,N}
\longrightarrow
\mathcal L_{\beta,i,N}|_{\Delta_Y}
\longrightarrow0
$$
and the preceding vanishing show that restriction to the diagonal is
surjective on global sections.  By the K\"unneth formula,
$$
H^0(Y\times Y,\mathcal L_{\beta,i,N})
\cong
H^0\left(Y,\sum_j\beta_jA_{j,N}\right)
\otimes H^0(Y,A_{i,N}),
$$
while, under the identification $\Delta_Y\cong Y$,
$\mathcal L_{\beta,i,N}|_{\Delta_Y} \cong \OO_Y\left(\sum_j\beta_jA_{j,N}+A_{i,N}\right)$.
Under these identifications, restriction sends a decomposable section
$s\boxtimes t$ to the product $st$.  Hence it is exactly the required
multiplication map, which is therefore surjective.

We argue by induction on the total multidegree.  The assertion is immediate in
degree one.  Let $\gamma\in\ZZ_{\geq0}^3$ have $|\gamma|\geq2$, choose
$i$ with $\gamma_i>0$, and put $\beta=\gamma-e_i$.  Then $\beta\neq0$,
and the surjective multiplication map above expresses every section of
$H^0\left(Y,\sum_j\gamma_jA_{j,N}\right)$
as a sum of products of a section of
$H^0\left(Y,\sum_j\beta_jA_{j,N}\right)$ and a section of
$H^0(Y,A_{i,N})$.  The induction hypothesis applies to the first factor.
Thus every summand of $H^0(X,|\gamma|L_N)$ is generated by products of
sections of degree one, and $S_N$ is generated in degree one.

It remains to prove normality.  Let
$\mathbf V(L_N^{-1}) :=\operatorname{Spec}_X\bigl(\operatorname{Sym} L_N\bigr)$
be the total space of $L_N^{-1}$.  Since $Y$ is integral, the projective bundle $X=\PP_Y(E)$ is integral.  Since $Y$ is smooth and $E$ is locally free, $X$ is smooth, hence regular and normal.  The total space $\mathbf V(L_N^{-1})$ is Zariski locally the spectrum of a polynomial ring over a normal domain, and is therefore integral and normal.  Its projection to $X$ is affine with direct image
$\operatorname{Sym}L_N =\bigoplus_{m\geq0}L_N^{\otimes m}$.
Since $X$ is quasi-compact and separated, global sections commute with this direct sum, so
$$
\Gamma\bigl(\mathbf V(L_N^{-1}),\OO\bigr)
=
\bigoplus_{m\geq0}H^0(X,mL_N)
=
R(X,L_N).
$$
The ring of global functions on an integral normal scheme is a normal
domain; see \cite[Tag~0358]{Stacks}.  Hence $S_N=R(X,L_N)$ is normal.
This part of the argument does not require $N$ to be large.
\end{proof}
Fix an integer $N$ for which Lemma~\ref{lem:LN-normal-standard} holds.  Since $X$ is integral and projective over the algebraically closed field $k$, one has $H^0(X,\OO_X)=k$.  Thus $S_N$ is a normal standard graded domain over $k$, and Subsection~\ref{subsec:section-rings-divisorial-ideals} gives
$X\cong\Proj S_N$ and $\OO_{\Proj S_N}(1)\cong L_N.$
Moreover, $\dim X=4$ and $\dim S_N=5$. From now on, $N$ is fixed and $Q=p^e$ tends to infinity. The later limit as $N\to\infty$ is used only to compare the sizes of the two logarithmic terms.  Their positivity for each fixed $N$ is proved in Propositions~\ref{prop:exact-lower-log} and~\ref{prop:upper-leading-log-small}. We retain the effective divisors $F_i=2E_i$ chosen in Lemma~\ref{lem:enriques-pencil-package}.  Thus
$S_Y=F_0+F_1+F_2$
is an effective ample Cartier divisor.  Put
$$
        D_N:=N\pi^*S_Y.
$$
Then $D_N$ is an effective Cartier divisor on $X$, and
$L_N-D_N=\xi$.
Let $s_{D_N}\in H^0(X,\OO_X(D_N))$ denote its canonical section.

The projective bundle formula gives
$$W:=H^0(Y,F_0)\oplus H^0(Y,F_1)\oplus H^0(Y,F_2)=H^0(Y,E)=H^0(X,\xi).$$
Since each $|F_i|$ is a pencil, $\dim_k W=6$.

The linear system $W\subseteq H^0(X,\xi)$ is base point free.  Indeed, let $x\in X=\PP_Y(E)$ lie over $y\in Y$.  The point $x$ corresponds to a one dimensional quotient
$E_y\twoheadrightarrow \xi_x$.
At least one of the summand maps
$\OO_Y(F_i)_y\longrightarrow \xi_x$
is nonzero.  Since $|F_i|$ is base point free, there is a section of $\OO_Y(F_i)$ whose value at $y$ is nonzero.  Its image in $\xi_x$ is therefore nonzero.

Multiplication by the nonzero canonical section $s_{D_N}$ defines the injective map
$$
\begin{aligned}
        H^0(X,L_N-D_N)=H^0(X,\xi)
        &\longrightarrow H^0(X,L_N),\\
        w&\longmapsto s_{D_N}w.
\end{aligned}
$$
We denote its image by
$$
        \widehat W:=s_{D_N}W\subseteq H^0(X,L_N)=(S_N)_1
$$
and define
$$
        J_N:=S_N\cdot\widehat W.
$$
Since $W$ globally generates $\xi$, the image sheaf of
$$
        \widehat W\otimes\OO_X(-L_N)
        \longrightarrow
        \OO_X
$$
is $\OO_X(-D_N)$.  Hence
$$
        \widetilde{J_N}=\OO_X(-D_N)
        \qquad\text{and}\qquad
        \Proj(S_N/J_N)=D_N.
$$
Since $D_N$ is a nonzero Cartier divisor on the fourfold $X$, it has dimension $3$.  Therefore
$$
        \dim(S_N/J_N)=4,
        \qquad
        \operatorname{ht}J_N=1.
$$

Apply Lemma~\ref{lem:frobenius-saturation-by-degree} with $L=L_N$, $D=D_N$, and $c=1$.  The global generation hypothesis is the base point freeness of $W=H^0(X,\xi)$, and the ideal in that lemma is $J_N=S_N\widehat W$.  Thus, for every $Q=p^e$,
$$
        (J_N^{[Q]})^{\rm sat}=\Gamma_*\OO_X(-QD_N),
        \qquad
        (J_N^{[Q]})_m=0 \quad(m<Q).
$$
The quotient $(J_N^{[Q]})^{\rm sat}/J_N^{[Q]}$ is a finitely generated $(S_N)_+$-torsion module and therefore has finite length.  We divide its graded pieces according to whether their degree is smaller than $Q$ or at least $Q$:
$$
\lambda_{S_N}\left(
        \frac{(J_N^{[Q]})^{\rm sat}}{J_N^{[Q]}}
        \right)
=
L_{N,Q}+U_{N,Q},
$$
where
$$
        L_{N,Q}
        :=
        \sum_{0\leq m<Q}
        h^0(X,mL_N-QD_N).
$$

To describe the remaining degrees, put
$W^{(Q)}:=W\otimes_{k,F^e}k$.
There is a natural $k$-linear map
$$
\operatorname{Fr}_{W,Q}:
        W^{(Q)}\longrightarrow H^0(X,Q\xi),
        \qquad
        w\otimes c\longmapsto c\,w^Q.
$$
The Frobenius twist is essential for $k$-linearity: the relation
$\lambda w\otimes c=w\otimes \lambda^Qc$
in $W^{(Q)}$ is compatible with $(\lambda w)^Q=\lambda^Qw^Q$.  This map is injective.  Indeed, choose a basis $w_1,\ldots,w_6$ of $W$, and suppose that
$\sum_i c_iw_i^Q=0$.
Since $k$ is algebraically closed, it is perfect, so we may write $c_i=d_i^Q$.  Then
$\left(\sum_i d_iw_i\right)^Q=0$.
Because $X$ is integral, Frobenius is injective on its local rings, and hence
$\sum_i d_iw_i=0$.
The linear independence of the $w_i$ gives $d_i=0$ for every $i$.  Thus $W^{(Q)}$ identifies with the six dimensional subspace of $H^0(X,Q\xi)$ generated by the $Q$th powers of the chosen basis sections.

For $a\geq0$, define
$$
        C_{Q,a}
        :=
        \operatorname{Coker}\left(
        W^{(Q)}\otimes H^0(X,aL_N)
        \longrightarrow
        H^0(X,Q\xi+aL_N)
        \right),
$$
where the map sends
$(w\otimes c)\otimes\sigma \longmapsto c\,w^Q\sigma$.
Set
$$
        U_{N,Q}:=\sum_{a\geq0}\dim_k C_{Q,a}.
$$

We identify $C_{Q,a}$ with the saturation defect in degree $Q+a$.  The ideal $J_N^{[Q]}$ is generated in degree $Q$ by the $Q$th powers of the elements of $\widehat W$.  Since an element of $\widehat W$ has the form $s_{D_N}w$, its $Q$th power is
$(s_{D_N}w)^Q=s_{D_N}^Qw^Q$.
Consequently,
$(J_N^{[Q]})_{Q+a}$
is the image of
$$
        W^{(Q)}\otimes H^0(X,aL_N)
        \longrightarrow
        H^0(X,(Q+a)L_N),
$$
given by
$w\otimes\sigma \longmapsto s_{D_N}^Qw^Q\sigma$.

On the other hand,
$$\bigl((J_N^{[Q]})^{\rm sat}\bigr)_{Q+a}=H^0\bigl(X,(Q+a)L_N-QD_N\bigr)=H^0(X,Q\xi+aL_N).$$
Multiplication by the fixed section $s_{D_N}^Q$ identifies this space with the subspace of
$H^0(X,(Q+a)L_N)$
consisting of sections vanishing along $QD_N$.  Under this identification, the piece of $J_N^{[Q]}$ of degree $Q+a$ corresponds exactly to the image of
$$
        W^{(Q)}\otimes H^0(X,aL_N)
        \longrightarrow
        H^0(X,Q\xi+aL_N).
$$
Therefore
$C_{Q,a} \cong \left( \frac{(J_N^{[Q]})^{\rm sat}}{J_N^{[Q]}} \right)_{Q+a}$.
Put
$R_N:=(S_N)_{(S_N)_+},  \mathfrak m_N:=(S_N)_+R_N$.
Combining the identification in Subsection~\ref{subsec:graded-length-saturation} with this decomposition by degree gives the exact identity for each finite $Q$
$$
\lambda_{R_N}\!\left(
H^0_{\mathfrak m_N}\!\left(R_N/J_N^{[Q]}R_N\right)
\right)
=
L_{N,Q}+U_{N,Q}.
$$

Thus $L_{N,Q}$ records the degrees $m<Q$, in which no Frobenius generator has yet appeared and the entire saturation piece is determined by the divisor $QD_N$.  The term $U_{N,Q}$ records the degrees $m=Q+a$, in which the Frobenius generators act.  In these degrees the sheaf map
$$
        W^{(Q)}\otimes\OO_X(aL_N)
        \longrightarrow
        \OO_X(Q\xi+aL_N)
$$
is surjective because the $Q$th powers of $W$ globally generate $Q\xi$.  The induced map on global sections may fail to be surjective, and $C_{Q,a}$ measures this failure.  Example~\ref{ex:veronese-lower-upper} shows that the upper range can have the same leading order as the range below $Q$.  We therefore cannot treat it as an error term.

\begin{lemma}\label{lem:upper-sum-finite}
For each fixed $N$ there is a constant $M_N$ such that
$C_{Q,a}=0$ for $a>M_NQ$.
Consequently $U_{N,Q}$ is a finite sum over $0\leq a\leq M_NQ$.
\end{lemma}

\begin{proof}
The evaluation map
$$
        W\otimes\OO_X\longrightarrow\OO_X(\xi)
$$
is surjective.  Since $L_N-D_N=\xi$, tensoring by
$\OO_X(-L_N)$ gives an exact sequence
$$
        0\longrightarrow \Kcal_N\longrightarrow
        W\otimes\OO_X(-L_N)
        \longrightarrow
        \OO_X(-D_N)
        \longrightarrow0.
$$
The variety $X$ is regular, so Frobenius is flat by Kunz's theorem.
After pulling back by $F^e$ and twisting by $(Q+a)L_N$, we obtain
$$
        0\longrightarrow
        F^{e*}\Kcal_N\otimes(Q+a)L_N
        \longrightarrow
        W^{(Q)}\otimes\OO_X(aL_N)
        \longrightarrow
        \OO_X(Q\xi+aL_N)
        \longrightarrow0.
$$
Thus $C_{Q,a}=0$ whenever
$$
        H^1\bigl(X,F^{e*}\Kcal_N\otimes(Q+a)L_N\bigr)=0.
$$
Lemma~\ref{lem:linear-frobenius-serre-vanishing}, applied to the fixed
sheaf $\Kcal_N$ and the ample line bundle $L_N$, gives a constant $C_N$
such that this vanishing holds once $Q+a\geq C_NQ$.  Choose an integer
$$
M_N\geq\max\{2,\lceil C_N\rceil\}.
$$
Then $a>M_NQ$ implies $C_{Q,a}=0$.
\end{proof}

\subsection{Uniform Cohomology in the Pencil Span}\label{subsec:uniform-pencil-cohomology}

We analyze the upper cokernels after pushing the multiplication maps to
$Y$.  The splitting
$E=\OO_Y(F_0)\oplus\OO_Y(F_1)\oplus\OO_Y(F_2)$
decomposes both source and target into monomial summands.  On each
summand, a Frobenius pullback of one of the pencil sequences controls the
cokernel through the cohomology of an explicit line bundle on $Y$.

Lemma~\ref{lem:upper-sum-finite} gives a constant $M_N$ such that
$C_{Q,a}=0$ for $a>M_NQ$.  Thus only the degrees
$0\leq a\leq M_NQ$ can contribute to $U_{N,Q}$.  For each such $a$,
the monomial decomposition is indexed by triples
$\beta\in\ZZ_{\geq0}^3$ with $|\beta|=Q+a$, so it has
$\binom{Q+a+2}{2}=O_N(Q^2)$ blocks.  Since there are $O_N(Q)$ possible
values of $a$, the upper sum has $O_N(Q^3)$ blocks altogether.  The
normalized parameters $a/Q$ and $\beta/Q$ remain in a compact set
which may depend on $N$ but not on $Q$; the same is therefore true of
the corresponding divisor classes in
$V=\spann_\RR\{F_0,F_1,F_2\}$.  We need an $o(Q^2)$ cohomology estimate
that is uniform on compact subsets of $V$.

\begin{lemma}\label{lem:uniform-enriques-pencil-cohomology}
Let
$V=\spann_\RR\{F_0,F_1,F_2\}\subseteq N^1(Y)_\RR$.
Fix a compact subset $K\subseteq V$, a finite set $\mathcal T$ of integral divisors numerically equivalent to zero on $Y$, and a constant $C>0$.  Let $D_Q$ be any sequence of integral divisors of the form
$$
        D_Q=n_{0,Q}F_0+n_{1,Q}F_1+n_{2,Q}F_2+T_Q,
        \qquad T_Q\in\mathcal T,
$$
such that, after setting
$$
        \alpha_Q=\frac{1}{Q}(n_{0,Q}F_0+n_{1,Q}F_1+n_{2,Q}F_2)\in V,
$$
one has
$\operatorname{dist}(\alpha_Q,K)\leq \frac{C}{Q}$
for all $Q$, with respect to any fixed norm on $V$.  Then
$$
        h^1(Y,\OO_Y(D_Q))=
        \frac{1}{2}\max\{-D_Q^2,0\}+o(Q^2).
$$
Equivalently,
$$
        h^1(Y,\OO_Y(D_Q))=
        \frac{Q^2}{2}\max\{-\alpha_Q^2,0\}+o(Q^2).
$$
Moreover,
$$
        h^0(Y,\OO_Y(D_Q))=
        \frac{Q^2}{2}\vol_Y(\alpha_Q)+o(Q^2),
$$
and
$$
        h^2(Y,\OO_Y(D_Q))=
        \frac{Q^2}{2}\vol_Y(-\alpha_Q)+o(Q^2).
$$
For fixed $K$, $\mathcal T$, and $C$, all three $o(Q^2)$ terms are uniform over the sequences satisfying the displayed distance bound.
\end{lemma}

\begin{proof}
The classes $F_0,F_1,F_2$ form a basis of $V$, since their intersection
matrix is nondegenerate.  The distance hypothesis therefore gives a
constant $C_0>0$, depending only on $K$ and $C$, such that
$$
        |n_{0,Q}|+|n_{1,Q}|+|n_{2,Q}|\leq C_0Q.
$$

Write $n_{i,Q}=n_{i,Q}^+-n_{i,Q}^-$ with $n_{i,Q}^{\pm}\geq0$.
For a fixed $T\in\mathcal T$, apply
\cite[Proposition~3.5.1]{BGGJKM} to the finite family
$$
        F_0,-F_0,F_1,-F_1,F_2,-F_2,T.
$$
We omit divisors whose coefficient is zero.  There are only finitely
many supports and finitely many choices of $T$, so one error function
$\rho(M)=o(M^2)$ works for all of them.  Here we may take
$$
        M_Q=1+\sum_{i=0}^2(n_{i,Q}^++n_{i,Q}^-),
$$
and $M_Q\leq1+C_0Q$.  The monotone envelope
$\widehat\rho(T)=\max_{1\leq M\leq T}|\rho(M)|$ also satisfies
$\widehat\rho(T)=o(T^2)$: above a fixed threshold this follows from
$\rho(M)=o(M^2)$, while the finitely many smaller values are bounded.
Hence, uniformly in the sequence and in $T_Q$,
$$
h^0(Y,\OO_Y(D_Q))
=
\frac{1}{2}\vol_Y(D_Q)+o(Q^2)
=
\frac{Q^2}{2}\vol_Y(\alpha_Q)+o(Q^2).
$$
The second equality uses numerical invariance and homogeneity of
volume.

By Serre duality,
$$
        h^2(Y,\OO_Y(D_Q))
        =h^0(Y,\OO_Y(K_Y-D_Q)).
$$
The divisors $K_Y$ and $T_Q$ are numerically trivial.  Applying the same
argument to the finite set $K_Y-\mathcal T$ gives
$$
        h^2(Y,\OO_Y(D_Q))
        =\frac{Q^2}{2}\vol_Y(-\alpha_Q)+o(Q^2)
$$
uniformly.

Lemma~\ref{lem:enriques-pencil-package} gives
$$
        \vol_Y(\alpha)+\vol_Y(-\alpha)=\max\{\alpha^2,0\}
        \qquad(\alpha\in V).
$$
Since $\chi(Y,\OO_Y)=1$ and $K_Y$ is numerically trivial,
Riemann--Roch gives
$$
        \chi(Y,\OO_Y(D_Q))=1+\frac{D_Q^2}{2},
        \qquad
        D_Q^2=Q^2\alpha_Q^2.
$$
Combining these formulas yields
$$
\begin{aligned}
        h^1(Y,\OO_Y(D_Q))
        &=h^0(Y,\OO_Y(D_Q))+h^2(Y,\OO_Y(D_Q))
          -\chi(Y,\OO_Y(D_Q))\\
        &=\frac{Q^2}{2}\max\{-\alpha_Q^2,0\}+o(Q^2)\\
        &=\frac{1}{2}\max\{-D_Q^2,0\}+o(Q^2),
\end{aligned}
$$
with the same uniformity.
\end{proof}

\begin{unnumberedremark}[Isotropic boundary]
The error in Lemma~\ref{lem:uniform-enriques-pencil-cohomology} need not stay bounded on an isotropic ray.  Let $f_i:Y\to\PP^1$ be the genus one fibration defined by $|F_i|$.  Its fibers are connected, and
$\OO_Y(F_i)\simeq f_i^*\OO_{\PP^1}(1)$.
Hence
$h^0(Y,\OO_Y(QF_i))=Q+1$.
The divisor $K_Y-QF_i$ has negative intersection with $S_Y$, so it is not effective.  Thus $h^2(Y,\OO_Y(QF_i))=0$.  Since $(QF_i)^2=0$, Riemann--Roch gives
$h^1(Y,\OO_Y(QF_i))=Q$.
The quadratic term in Lemma~\ref{lem:uniform-enriques-pencil-cohomology} therefore vanishes on this ray, while $h^1$ grows linearly.  Thus no uniform $O(1)$ error can hold on compact sets that meet the isotropic boundary.  The uniform $o(Q^2)$ estimate is nevertheless sufficient for the leading term used below.
\end{unnumberedremark}

\subsection{The Lower Degree Range and Its Logarithmic Term}\label{subsec:split-pencil-lower}

For $0\leq m<Q$, the saturation quotient is the full space $H^0(X,mL_N-QD_N)$.  Lemma~\ref{lem:uniform-volume-asymptotics} converts the normalized sum of these dimensions into an integral in one variable of $\vol_X(tL_N-D_N)$.  Its interval of positivity is determined by the pseudo-effective threshold.  By Lemma~\ref{lem:weighted-pb-volume}, the divisor $\xi-u\pi^*S_Y$ has positive volume precisely when the set of points $\mu\in\Delta_1$ for which
$\sum_{i=0}^2(\mu_i-u)F_i$
lies in the positive cone has positive measure.  Write
$x_i=\mu_i-\frac{1}{3}$, $ r=\frac{1}{3}-u$ and $R^2=x_0^2+x_1^2+x_2^2$.
Then $x_0+x_1+x_2=0$ and
$$
        \left(\sum_{i=0}^2(r+x_i)F_i\right)^2
        =\eta(6r^2-R^2).
$$
Moreover, since $F_i\cdot S_Y=2\eta$ for every $i$,
$\left(\sum_{i=0}^2(r+x_i)F_i\right)\cdot S_Y =6\eta r$.
The class therefore lies in the closed positive cone containing $S_Y$ exactly when $r\geq0$ and $R\leq\sqrt6\,r$.  If $u<1/3$, then $r>0$ and both inequalities are strict at the center of the simplex; hence they remain strict on a neighborhood of positive measure.  Thus $\xi-u\pi^*S_Y$ has positive volume and is big for $u<1/3$.  At $u=1/3$, the admissible locus is only the center of the simplex, so the volume is zero.

We claim that no class with $u>1/3$ is pseudo-effective.  Suppose that
$D_u:=\xi-u\pi^*S_Y$
were pseudo-effective for some $u>1/3$.  Since $L_N$ is ample, $D_u+\delta L_N$ would be big for every $\delta>0$.  But
$$
D_u+\delta L_N
=(1+\delta)\left(
\xi-\frac{u-\delta N}{1+\delta}\pi^*S_Y
\right).
$$
For $\delta>0$ sufficiently small, the coefficient
$\frac{u-\delta N}{1+\delta}$
is still larger than $1/3$.  For every point of $\Delta_1$, the corresponding class on $Y$ then has negative intersection with $S_Y$, so the projective bundle volume formula gives volume zero, a contradiction.  On the other hand, the classes with $u<1/3$ are big and converge to the class with $u=1/3$; since the pseudo-effective cone is closed, the boundary class is pseudo-effective.  Hence the pseudo-effective threshold of the ray $\xi-u\pi^*S_Y$ is exactly $u=1/3$.

For the ray $tL_N-D_N$, the change of variables
$u=\frac{N(1-t)}{t}$ and $t=\frac{N}{N+u}$
shows that the threshold is
$\tau_N=\frac{N}{N+1/3}=\frac{3N}{3N+1}$.
If $m<\lceil\tau_NQ\rceil$, then $(m/Q)L_N-D_N$ is not pseudo-effective.  A nonzero section of $mL_N-QD_N$ would make this class effective, which is impossible.  Hence, for each finite $Q$, the lower sum is exactly
$$
        L_{N,Q}
        =\sum_{\lceil\tau_NQ\rceil\leq m<Q}
        h^0(X,mL_N-QD_N).
$$
By Lemma~\ref{lem:uniform-volume-asymptotics}, there is a function $\rho_N(Q)=o(Q^4)$ such that, uniformly for $\lceil\tau_NQ\rceil\leq m\leq Q$,
$$
\left|
 h^0(X,mL_N-QD_N)
 -\frac{Q^4}{24}\,
  \vol_X\left(\frac{m}{Q}L_N-D_N\right)
\right|
\leq \rho_N(Q).
$$
There are $O(Q)$ terms, so the accumulated replacement error is
$O(Q)\rho_N(Q)=o(Q^5)$.
The remaining normalized sum is the Riemann sum of the continuous volume function.  Therefore
$$
        \frac{L_{N,Q}}{Q^5}
        \longrightarrow
        \mathcal L_N
        :=\frac{1}{24}\int_{\tau_N}^1
        \vol_X(tL_N-D_N)\,dt.
$$
The proposition below gives the logarithmic part of $\mathcal L_N$, its sign, and its leading term as $N\to\infty$.

\begin{proposition}\label{prop:exact-lower-log}
Put
$$
        \gamma_N=\sqrt{\frac{6N+1}{6N+3}},
        \qquad
        \alpha_N=
        \frac{\gamma_N+1/\sqrt3}{\gamma_N-1/\sqrt3},
$$
and
$$
        b_N=
        \frac{\sqrt3(1-\gamma_N^2)(3\gamma_N^2-1)^5}
        {5760\gamma_N^5(\gamma_N^2+1)}.
$$
Then $\alpha_N>1$, $b_N>0$, and
$\mathcal L_N=A_N^{\rm low}+b_N\log\alpha_N$
for some $A_N^{\rm low}\in\overline{\QQ}$.  Moreover,
$$
        b_N\log\alpha_N
        =
        \frac{\sqrt3}{1080N}\log(2+\sqrt3)
        +O\left(\frac{1}{N^2}\right).
$$
In particular, the logarithmic part of $\mathcal L_N$ is positive and has order $N^{-1}$.
\end{proposition}

\begin{proof}
The change from the simplex integral to centered radial coordinates, the rational primitive, and the residue calculation are carried out in Proposition~\ref{app:lower-log-computation}.  That computation gives
$\mathcal L_N=A_N^{\rm low}+b_N\log\alpha_N$
with $A_N^{\rm low}\in\overline{\QQ}$, $b_N>0$, and $\alpha_N>1$.  The expansion recorded in \eqref{eq:app-lower-asymptotic} gives
$$
        b_N\log\alpha_N
        =
        \frac{\sqrt3}{1080N}\log(2+\sqrt3)
        +O(N^{-2}),
$$
which proves the proposition.
\end{proof}

\subsection{The Upper Degree Range}\label{subsec:actual-upper-cokernel}

For $m=Q+a$, the saturation defect is the cokernel $C_{Q,a}$.  Put
$s=a/Q$.  By the projective bundle formula and the projection formula,
$$
        H^0(X,Q\xi+aL_N)
        =
        H^0\left(Y,\Sym^{Q+a}E\otimes\OO_Y(aNS_Y)\right).
$$
Since $E$ is split,
$$
        \Sym^mE
        =\bigoplus_{|\alpha|=m}
        \OO_Y(\alpha_0F_0+\alpha_1F_1+\alpha_2F_2),
$$
with one summand for each exponent triple.  Hence
$$
        H^0(X,aL_N)
        =\bigoplus_{|\alpha|=a}
        H^0\left(Y,
        \OO_Y\left(\sum_j\alpha_jF_j+aNS_Y\right)\right),
$$
while the target is the direct sum over
$\beta=(\beta_0,\beta_1,\beta_2)\in\ZZ_{\geq0}^3$ with
$|\beta|=Q+a$ of the spaces
$H^0(Y,\OO_Y(D_{\beta,a}))$, where
$$
        D_{\beta,a}=(\beta_0+aN)F_0+
        (\beta_1+aN)F_1+
        (\beta_2+aN)F_2.
$$

Write $W_i=H^0(Y,F_i)$, so that $W=\bigoplus_iW_i$.  Since $Q=p^e$,
$$
        (w_0+w_1+w_2)^Q=w_0^Q+w_1^Q+w_2^Q
        \qquad(w_i\in W_i).
$$
Thus the $i$th Frobenius generator, multiplied by the source summand
indexed by $\alpha$, lands only in the target summand indexed by
$\beta=\alpha+Qe_i$.  It contributes to the $\beta$ block exactly when
$\beta_i\geq Q$.

\begin{lemma}\label{lem:upper-cokernel-bookkeeping}
Let
$H(\beta)=\{i\in\{0,1,2\}\mid\beta_i\geq Q\}$.  Then
$$
        C_{Q,a}=\bigoplus_{|\beta|=Q+a}C_{\beta,a},
$$
where $C_{\beta,a}$ is the cokernel of
$$
        \bigoplus_{i\in H(\beta)}
        H^0(Y,F_i)^{(Q)}\otimes
        H^0(Y,\OO_Y(D_{\beta,a}-QF_i))
        \longrightarrow
        H^0(Y,\OO_Y(D_{\beta,a})).
$$
All the error terms below are uniform when $\beta/Q$ and $a/Q$ vary in
a fixed compact set.  The blocks fall into the following three cases.
\begin{enumerate}
\item[(i)] If $H(\beta)=\varnothing$, then
$C_{\beta,a}=H^0(Y,\OO_Y(D_{\beta,a}))$ and
$$
        \dim_kC_{\beta,a}
        =\frac{Q^2}{2}\left(\frac{D_{\beta,a}}{Q}\right)^2+o(Q^2).
$$
\item[(ii)] If $H(\beta)\neq\varnothing$ and $s\geq1/N$, then
$$
        \dim_kC_{\beta,a}=o(Q^2).
$$
\item[(iii)] If $H(\beta)\neq\varnothing$ and $s<1/N$, then
$H(\beta)=\{i\}$ is a singleton and
$$
        \dim_kC_{\beta,a}
        =h^1(Y,\OO_Y(D_{\beta,a}-2QF_i))+o(Q^2).
$$
\end{enumerate}
\end{lemma}

\begin{proof}
The source is indexed by pairs $(i,\alpha)$ with $|\alpha|=a$, and the
summand $(i,\alpha)$ maps only to $\beta=\alpha+Qe_i$.  The
multiplication map is therefore block diagonal, with the blocks stated
above.

If $H(\beta)=\varnothing$, no source summand reaches the $\beta$ block.
The divisor $D_{\beta,a}$ is a nonnegative linear combination of the
nef classes $F_0,F_1,F_2$.  Lemma~\ref{lem:uniform-enriques-pencil-cohomology},
applied to $h^0$, gives (i).

Assume that $H(\beta)\neq\varnothing$ and $s\geq1/N$, and choose
$i\in H(\beta)$.  Retaining only the $i$th source summand gives a
cokernel that surjects onto $C_{\beta,a}$.  Frobenius pullback of the
pencil sequence, followed by a twist by
$\OO_Y(D_{\beta,a}-QF_i)$, gives
$$
        0\longrightarrow\OO_Y(D_{\beta,a}-2QF_i)
        \longrightarrow
        H^0(Y,F_i)^{(Q)}\otimes\OO_Y(D_{\beta,a}-QF_i)
        \longrightarrow
        \OO_Y(D_{\beta,a})
        \longrightarrow0.
$$
The coefficient of $F_i$ in $Q^{-1}(D_{\beta,a}-2QF_i)$ is at least
$1+Ns-2\geq0$, and all other coefficients are nonnegative.  Thus
$D_{\beta,a}-2QF_i$ is nef.  The long exact cohomology sequence and
Lemma~\ref{lem:uniform-enriques-pencil-cohomology} give an $o(Q^2)$
bound for the one-pencil cokernel, hence also for $C_{\beta,a}$.  This
proves (ii).

Finally, assume that $H(\beta)\neq\varnothing$ and $s<1/N$.  Then
$H(\beta)=\{i\}$: two high coordinates would give
$Q(1+s)=|\beta|\geq2Q$, contrary to $s<1/N\leq1$.  The preceding exact
sequence now describes the full block.  Its cokernel on global sections
is the kernel of
$$
H^1(Y,\OO_Y(D_{\beta,a}-2QF_i))
\longrightarrow
H^0(Y,F_i)^{(Q)}\otimes
H^1(Y,\OO_Y(D_{\beta,a}-QF_i)).
$$
The divisor $D_{\beta,a}-QF_i$ is nef: its coefficient of $F_i$ is
$\beta_i-Q+aN\geq0$, and all other coefficients are nonnegative.
Lemma~\ref{lem:uniform-enriques-pencil-cohomology} therefore gives
$$
        h^1(Y,\OO_Y(D_{\beta,a}-QF_i))=o(Q^2)
$$
uniformly.  Since $h^0(Y,F_i)=2$, the dimension of the displayed kernel
differs from $h^1(Y,\OO_Y(D_{\beta,a}-2QF_i))$ by $o(Q^2)$.  This proves
(iii).
\end{proof}

\begin{remark}\label{rem:upper-bookkeeping-meaning}
The comparison in Lemma~\ref{lem:upper-cokernel-bookkeeping} comes from
the long exact cohomology sequence of a Frobenius pullback of a pencil
sequence.  The possible quadratic term is the $H^1$ of the kernel line
bundle $\OO_Y(D_{\beta,a}-2QF_i)$; it is not obtained by replacing the
multiplication cokernel with a sheaf quotient.
\end{remark}

The integral over the shrinking exceptional box has the arithmetic form and sign stated below.

\begin{lemma}\label{lem:critical-integral-arithmetic}
Let
$\Delta=\{(u,w)\in\RR^2\mid u\geq0,\ w\geq0,\ u+w\leq1\}$.
For $N\geq1$, put $\epsilon=1/N$ and define
$A_\epsilon(u,w) = 3+2\epsilon -\epsilon^2(u^2+uw-u+w^2-w)$,
and
$B_\epsilon(u)=2+\epsilon(1-u)$.
Set
$$
        I_N=
        \int_\Delta
        \frac{B_\epsilon(u)^5}{A_\epsilon(u,w)^4}
        \,du\,dw.
$$
Then $I_N$ has a presentation by algebraic numbers and logarithms.  Its logarithmic part is
$$
\Lambda(I_N)
=
-\frac{4D_\epsilon
(\epsilon^5+10\epsilon^4+40\epsilon^3+80\epsilon^2+80\epsilon+30)}
{\epsilon^2(\epsilon+2)^3(\epsilon+3)(\epsilon+6)^3}
\log\left(\frac{D_\epsilon-\epsilon}{D_\epsilon+\epsilon}\right),
$$
where
$D_\epsilon=\sqrt{\epsilon^2+8\epsilon+12}$.
Moreover,
$\Lambda(I_N)>0$.
In particular,
$\Lambda(I_N)=\frac{5}{108\epsilon}+O(1)=O(N)$.
Consequently, since $\eta=4$, the logarithmic part of
$\frac{3\eta}{20N^3}I_N$
is
$\frac{1}{36N^2}+O(N^{-3})$.
\end{lemma}

\begin{proof}
Proposition~\ref{app:upper-log-computation} gives a divergence identity on $\Delta$, reduces the boundary integral to a rational function in one variable, and computes the coefficients of the simple poles at the two roots of the quadratic polynomial whose cube is the denominator.  The resulting logarithmic part is exactly the expression displayed above.  Since
$D_\epsilon^2-\epsilon^2=8\epsilon+12>0$,
we have $D_\epsilon>\epsilon>0$ and hence
$0<\frac{D_\epsilon-\epsilon}{D_\epsilon+\epsilon}<1$.
Write the displayed expression as $-P_\epsilon\log r_\epsilon$, where $P_\epsilon>0$ and $0<r_\epsilon<1$.  It follows that $\Lambda(I_N)>0$.  Equation~\eqref{eq:app-upper-asymptotic} gives
$$
        \Lambda(I_N)
        =\frac{5}{108\epsilon}+\frac{5}{324}+O(\epsilon),
$$
which implies both asserted estimates after multiplication by $3\eta/(20N^3)$.
\end{proof}

The three cases of Lemma~\ref{lem:upper-cokernel-bookkeeping} give the
corresponding decomposition of the normalized exponent space.  The
boundary faces $s=1/N$ and $\beta_i=Q$ contain only $O_N(Q^2)$ lattice
points and therefore contribute $O_N(Q^4)=o_N(Q^5)$.  Thus only case
(iii) can contribute a nonpolynomial leading term.

\begin{proposition}\label{prop:upper-leading-log-small}
For each sufficiently large fixed $N$, the limit
$$
        \lim_{Q=p^e\to\infty}\frac{U_{N,Q}}{Q^5}
$$
exists and can be written as an algebraic number plus a logarithmic part
$$
        A_N^{\rm up}+\Lambda_N^{\rm up},
        \qquad
        A_N^{\rm up}\in\overline{\QQ},
$$
where $\Lambda_N^{\rm up}$ is a finite linear combination over $\overline{\QQ}$ of logarithms of positive algebraic numbers.  In fact,
$$
        \Lambda_N^{\rm up}>0,
        \qquad
        \Lambda_N^{\rm up}=\frac{1}{36N^2}+O(N^{-3}).
$$
\end{proposition}

\begin{proof}
Fix $N$ and let $Q=p^e\to\infty$.  We first prove that the normalized
upper sum converges.  By Lemma~\ref{lem:upper-sum-finite},
$$
        U_{N,Q}
        =
        \sum_{a=0}^{\lfloor M_NQ\rfloor}
        \ \sum_{\substack{\beta\in\ZZ_{\geq0}^3\\|\beta|=Q+a}}
        \dim_k C_{\beta,a}.
$$
Put $s=a/Q$ and $b_i=\beta_i/Q$.  The parameters lie in the compact
polytope
$$
\mathcal K_N=
\left\{(s,b_0,b_1,b_2)\in\RR_{\geq0}^4\ \middle|\
0\leq s\leq M_N,\quad b_0+b_1+b_2=1+s\right\}.
$$
There are $O_N(Q)$ choices for $a$ and $O_N(Q^2)$ triples $\beta$ for
each $a$, hence $O_N(Q^3)$ blocks altogether.  By
Lemma~\ref{lem:upper-cokernel-bookkeeping}, there is a function
$\varepsilon_N(Q)\to0$ such that every blockwise error on
$\mathcal K_N$ is bounded by $\varepsilon_N(Q)Q^2$.  Their total is
therefore
$$
        O_N(Q^3)\varepsilon_N(Q)Q^2=o_N(Q^5).
$$
The same uniform estimates show that every block has dimension
$O_N(Q^2)$.  Each boundary face $s=1/N$ or $b_i=1$ contains only
$O_N(Q^2)$ lattice points, so all these faces contribute
$O_N(Q^4)=o_N(Q^5)$.  We may therefore work away from the boundaries.

We divide the remaining blocks into three regions.  First suppose that
$H(\beta)=\varnothing$.  Set
$$
        f_N(s,b_0,b_1,b_2)
        =\frac12\left(\sum_{i=0}^2(b_i+Ns)F_i\right)^2.
$$
Lemma~\ref{lem:upper-cokernel-bookkeeping}(i) gives
$$
        \dim_k C_{\beta,a}
        =Q^2f_N(s,b_0,b_1,b_2)+o_N(Q^2)
$$
uniformly.  The condition $H(\beta)=\varnothing$ is $b_i<1$ for every
$i$.  Hence, after division by $Q^5$, these blocks form a
three-dimensional Riemann sum over the rational polytope
$$
\mathcal P_N^{(0)}
=
\left\{
(s,b_0,b_1,b_2)\in\RR_{\geq0}^4
\ \middle|\
b_0+b_1+b_2=1+s,\quad b_i\leq1\ \text{for all }i
\right\}.
$$
The inequalities $b_i\leq1$ imply $s\leq2\leq M_N$, so this polytope
lies in $\mathcal K_N$ and does not depend on the auxiliary cutoff.
The corresponding Riemann sums converge to
$$
        A_N^{(0)}
        :=\int_{\mathcal P_N^{(0)}}f_N\,d\mu,
$$
where $d\mu=ds\,db_0\,db_1$ and
$b_2=1+s-b_0-b_1$.  Since $f_N$ is a polynomial with algebraic
coefficients and $\mathcal P_N^{(0)}$ is rational,
$A_N^{(0)}\in\overline{\QQ}$.

Next suppose that $H(\beta)\neq\varnothing$ and $s\geq1/N$.
Lemma~\ref{lem:upper-cokernel-bookkeeping}(ii) gives
$\dim_k C_{\beta,a}=o_N(Q^2)$ uniformly throughout this region.  Since
there are $O_N(Q^3)$ blocks, their total contribution is $o_N(Q^5)$.

It remains to consider case (iii) of
Lemma~\ref{lem:upper-cokernel-bookkeeping}.  By symmetry, assume that
the unique high coordinate is $\beta_0$.  Write
$$
        \beta_0=Q+\ell,\qquad
        v=\frac{\ell}{Q},\qquad
        x=\frac{\beta_1}{Q},
$$
so that
$$
        \frac{\beta_2}{Q}=s-v-x,
        \qquad
        0\leq v\leq s,
        \qquad
        0\leq x\leq s-v.
$$
Put
$$
        K_{\beta,a}=D_{\beta,a}-2QF_0,
        \qquad
        \frac1QK_{\beta,a}=c_0F_0+c_1F_1+c_2F_2.
$$
Then
$$
        c_0=-1+v+Ns,\qquad
        c_1=x+Ns,\qquad
        c_2=s-v-x+Ns.
$$
Make the change of variables
$$
        s=\frac TN,\qquad
        v=\frac{Tu}{N},\qquad
        x=\frac{Tw}{N},
$$
where $0\leq T\leq1$ and
$(u,w)\in\Delta$.  Its Jacobian is
$$
        ds\,dv\,dx=\frac{T^2}{N^3}\,dT\,du\,dw.
$$
Let $\epsilon=1/N$ and put $z=1-u-w$.  Then
$$
        c_0=-1+T+\epsilon Tu,\qquad
        c_1=T+\epsilon Tw,\qquad
        c_2=T+\epsilon Tz.
$$
A direct calculation gives
$$
        c_0c_1+c_0c_2+c_1c_2
        =T\bigl(A_\epsilon(u,w)T-B_\epsilon(u)\bigr),
$$
where $A_\epsilon$ and $B_\epsilon$ are those of
Lemma~\ref{lem:critical-integral-arithmetic}.  Moreover,
$$
        A_\epsilon
        =3+2\epsilon+\epsilon^2(uw+uz+wz)>0
$$
and
$$
        A_\epsilon-B_\epsilon
        =1+\epsilon(1+u)+\epsilon^2(uw+uz+wz)>0.
$$
Thus $0<B_\epsilon/A_\epsilon<1$ on $\Delta$.

Since $F_i^2=0$ and $F_i\cdot F_j=\eta$ for $i\neq j$,
Lemmas~\ref{lem:upper-cokernel-bookkeeping}(iii) and
\ref{lem:uniform-enriques-pencil-cohomology} give
$$
        \dim_k C_{\beta,a}
        =\eta Q^2
        \max\{-T(A_\epsilon T-B_\epsilon),0\}+o_N(Q^2)
$$
uniformly on this region.  The variables $(a,\ell,\beta_1)$ vary on a
lattice of mesh $1/Q$ in $(s,v,x)$.  Thus the normalized sum over this
region is a three-dimensional Riemann sum.  Accounting for the three
possible high coordinates, its limit is
$$
        \frac{3\eta}{N^3}
        \int_\Delta\int_0^1
        T^2\max\{T(B_\epsilon-A_\epsilon T),0\}
        \,dT\,du\,dw.
$$
The integrand is positive precisely for
$0\leq T<B_\epsilon/A_\epsilon$, and
$$
        \int_0^{B_\epsilon/A_\epsilon}
        T^3(B_\epsilon-A_\epsilon T)\,dT
        =\frac{B_\epsilon^5}{20A_\epsilon^4}.
$$
Hence the exceptional region contributes
$$
        \frac{3\eta}{20N^3}I_N.
$$
We have proved both the existence of the limit and the identity
$$
        \lim_{Q=p^e\to\infty}\frac{U_{N,Q}}{Q^5}
        =A_N^{(0)}+\frac{3\eta}{20N^3}I_N.
$$

By Lemma~\ref{lem:critical-integral-arithmetic}, $I_N$ is an algebraic
number plus a finite $\overline{\QQ}$-linear combination of logarithms
of positive algebraic numbers.  Absorbing its algebraic part into
$A_N^{(0)}$, we obtain
$$
        \lim_{Q=p^e\to\infty}\frac{U_{N,Q}}{Q^5}
        =A_N^{\rm up}+\Lambda_N^{\rm up},
        \qquad
        \Lambda_N^{\rm up}
        =\frac{3\eta}{20N^3}\Lambda(I_N)>0.
$$
Finally, $\eta=4$, and the same lemma gives
$$
        \Lambda_N^{\rm up}
        =\frac{1}{36N^2}+O(N^{-3}).
$$
\end{proof}

\subsection{The Generalized Hilbert--Kunz Value}\label{subsec:ghk-conclusion}

We now combine the exact degree decomposition with the two limiting
calculations.  Since both logarithmic contributions are positive, they
cannot cancel.  Their relative orders are not needed for transcendence.

\begin{theorem}\label{thm:enriques-gHK-criterion}
Let $k$ be an uncountable algebraically closed field of characteristic $p>2$.  For all sufficiently large integers $N$, the homogeneous ideal
$J_N=S_N\cdot\widehat W$
in the normal standard graded domain
$S_N=R(X,L_N)$
has transcendental generalized Hilbert--Kunz multiplicity:
$e_{\rm gHK}(J_N(S_N)_{(S_N)_+})$
is transcendental.
\end{theorem}

\begin{proof}
Fix $N$ sufficiently large and put
$R_N=(S_N)_{(S_N)_+}$ and $\mathfrak m_N=(S_N)_+R_N$.  Since
$\dim R_N=5$, the exact identity preceding
Lemma~\ref{lem:upper-sum-finite} gives
$$
\frac{
\lambda_{R_N}\!\left(
H^0_{\mathfrak m_N}(R_N/J_N^{[Q]}R_N)
\right)}{Q^5}
=
\frac{L_{N,Q}}{Q^5}+
\frac{U_{N,Q}}{Q^5}.
$$
The calculation in Subsection~\ref{subsec:split-pencil-lower} gives the
lower limit, and Proposition~\ref{prop:upper-leading-log-small} gives
the upper limit.  Hence $e_{\rm gHK}(J_NR_N)$ exists.

By Proposition~\ref{prop:exact-lower-log}, the lower limit has the form
$$
        A_N^{\rm low}+\Lambda_N^{\rm low},
        \qquad
        \Lambda_N^{\rm low}=b_N\log\alpha_N>0.
$$
By Proposition~\ref{prop:upper-leading-log-small}, the upper limit has
the form
$$
        A_N^{\rm up}+\Lambda_N^{\rm up},
        \qquad
        \Lambda_N^{\rm up}>0,
$$
where $A_N^{\rm low}$ and $A_N^{\rm up}$ are algebraic and both
logarithmic parts are algebraic linear combinations of logarithms of
positive algebraic numbers.  Combining them, we obtain positive
algebraic numbers $\rho_{N,1},\ldots,\rho_{N,r}$ and coefficients
$c_{N,1},\ldots,c_{N,r}\in\overline{\QQ}$ such that
$$
        e_{\rm gHK}(J_NR_N)
        =A_N+\sum_{j=1}^r c_{N,j}\log\rho_{N,j},
        \qquad A_N\in\overline{\QQ},
        \qquad
        \sum_{j=1}^r c_{N,j}\log\rho_{N,j}
        =\Lambda_N^{\rm low}+\Lambda_N^{\rm up}>0.
$$
Corollary~\ref{cor:Baker-linear-form} therefore shows that
$e_{\rm gHK}(J_NR_N)$ is transcendental.
\end{proof}

\section{From Generalized to Ordinary Hilbert--Kunz Multiplicity}\label{sec:ordinary-HK}

Section~\ref{sec:enriques-gHK} produces a transcendental generalized Hilbert--Kunz multiplicity for a homogeneous ideal that is not primary to the homogeneous maximal ideal.  We pass to ordinary Hilbert--Kunz multiplicity through Vraciu's reduction \cite[Theorem~2.1]{Vraciu}.  The induction introduces $Q$-dependent colon ideals rather than Frobenius powers of fixed ideals.  Observation~2.8 of \cite{Vraciu} reduces the required LC bounds for these families to LC for Frobenius powers of fixed ideals.  We establish those bounds geometrically and verify that every branch of the local induction arises by localizing homogeneous data.

Fix a sufficiently large $N$ for which Theorem~\ref{thm:enriques-gHK-criterion} holds, and write
$S=S_N=R(X,L_N)$, $J=J_N$, $ R=S_{S_+}$ and $\mathfrak m=S_+R$.
The ring $S$ is a normal standard graded domain, and $X=\Proj S$ is smooth.  The finiteness hypotheses used in \cite{Vraciu} are satisfied: the algebraically closed field $k$ is perfect and $F$-finite, and therefore the finitely generated $k$-algebra $S$ and its localization $R$ are $F$-finite.  The LC bounds and the homogeneous prime avoidance choices needed for the reduction are established below.  The residue field of $R$ is the uncountable field $k$.

For a family of ideals $\{A_Q\}_{Q=p^e}$ in $R$, the LC condition means that there is a constant $C$ such that
$$
        \mathfrak m^{CQ}H^0_{\mathfrak m}(R/A_Q)=0
        \qquad\text{for every }Q=p^e.
$$

\begin{lemma}\label{lem:fixed-homogeneous-LC}
Let $A\subseteq S$ be a fixed homogeneous ideal.  Then the Frobenius family $\{A^{[Q]}R\}_{Q=p^e}$ satisfies LC.  More precisely, there is an integer $B_A>0$ such that
$$
        \mathfrak m^{B_AQ}H^0_{\mathfrak m}(R/A^{[Q]}R)=0
        \qquad\text{for every }Q=p^e.
$$
\end{lemma}

\begin{proof}
If $A=0$, then $H^0_{S_+}(S/A^{[Q]})=H^0_{S_+}(S)=0$ because $S$ is a domain of positive dimension, and the assertion is immediate.  We may therefore assume that $A\neq0$.

Choose homogeneous generators $g_1,\ldots,g_t$ of $A$, with $\deg g_i=d_i$.  On $X=\Proj S$ they give an exact sequence
$$
        0\longrightarrow\Kcal_A\longrightarrow
        \bigoplus_{i=1}^t\OO_X(-d_iL_N)
        \longrightarrow\widetilde A\longrightarrow0,
$$
where $\widetilde A$ denotes the image ideal sheaf.  The variety $X$ is smooth over the perfect field $k$, so Frobenius is flat.  The pulled-back map is generated by $g_1^Q,\ldots,g_t^Q$; its image is therefore the ideal sheaf $\widetilde{A^{[Q]}}$.  Thus
$$
        0\longrightarrow F^{e*}\Kcal_A\longrightarrow
        \bigoplus_{i=1}^t\OO_X(-Qd_iL_N)
        \longrightarrow\widetilde{A^{[Q]}}\longrightarrow0
$$
is exact.
For every integer $m\geq0$, the piece of the saturation quotient of degree $m$ is the cokernel of
$$
        \bigoplus_{i=1}^tH^0(X,(m-Qd_i)L_N)
        \longrightarrow
        H^0(X,\widetilde{A^{[Q]}}\otimes mL_N);
$$
here we use Lemma~\ref{lem:full-section-ring-saturation}, with the convention that $H^0(X,rL_N)=0$ for $r<0$.  The cokernel vanishes if
$H^1(X,F^{e*}\Kcal_A\otimes mL_N)=0$.
By Lemma~\ref{lem:linear-frobenius-serre-vanishing}, there is a constant $C_A$ for which this holds whenever $m\geq C_AQ$.  Hence
$H^0_{S_+}(S/A^{[Q]})_m=0$ for $m\geq C_AQ$.
The module on the left has no part in negative degrees.  Since $S$ is standard graded, any integer $B_A>C_A+1$ satisfies
$(S_+)^{B_AQ}H^0_{S_+}(S/A^{[Q]})=0$.
Zeroth local cohomology commutes with localization here: if $T=S\setminus S_+$, then
$T^{-1}H^0_{S_+}(S/A^{[Q]}) \cong H^0_{\mathfrak m}(R/A^{[Q]}R)$.
Localizing this annihilation at $S_+$ gives the asserted LC bound.
\end{proof}

The proof gives a uniform linear annihilation bound.  The elementary calculation below shows that the linear dependence on $Q$ can be sharp and also exhibits the type of finite integer combination produced by the general reduction.

\begin{example}\label{ex:one-step-vraciu}
Let
$T=k[x,y]$, $\mathfrak n=(x,y)$ and  $A=(x^2,xy)$.
For every $Q=p^e$,
$A^{[Q]}=(x^{2Q},x^Qy^Q)$,  $(A^{[Q]})^{\rm sat}=(x^Q)$.
Therefore
$$
        H^0_{\mathfrak n}(T/A^{[Q]})
        \cong
        \frac{(x^Q)}{(x^{2Q},x^Qy^Q)}
        \cong
        \frac{T}{(x^Q,y^Q)}(-Q),
$$
and hence
$\lambda_T\bigl(H^0_{\mathfrak n}(T/A^{[Q]})\bigr)=Q^2$.
Moreover,
$\mathfrak n^{2Q-1}H^0_{\mathfrak n}(T/A^{[Q]})=0$,
whereas $\mathfrak n^{2Q-2}$ does not annihilate this module.  Thus the linear dependence on $Q$ in the LC bound is already sharp in this example.

Set
$B_1=A+(y)=(x^2,y)$,  $B_2=A+(y^2)=(x^2,xy,y^2)$.
Both ideals are $\mathfrak n$-primary, and direct monomial counts give
$\lambda_T(T/B_1^{[Q]})=2Q^2$,  $\lambda_T(T/B_2^{[Q]})=3Q^2$.
Consequently, for every $Q=p^e$,
$$
        \lambda_T\bigl(H^0_{\mathfrak n}(T/A^{[Q]})\bigr)
        =2\lambda_T(T/B_1^{[Q]})
        -\lambda_T(T/B_2^{[Q]}).
$$
The same identity holds after localization at $\mathfrak n$.  Thus the generalized Hilbert--Kunz function of a homogeneous ideal that is not primary is expressed, with coefficients independent of $Q$, as an integer linear combination of the ordinary Hilbert--Kunz functions of two fixed primary ideals.\hfill\qedsymbol
\end{example}

The next elementary lemma records the graded localization facts used in the reduction.

\begin{lemma}\label{lem:homogeneous-localization}
Let $R=S_{S_+}$.  Let $B,C\subseteq S$ be homogeneous ideals and let $y\in S$ be homogeneous.  Then sums, products, Frobenius powers, and colons by $y$ commute with localization.  In particular,
$(B:y)R=(BR:y/1)$.
Moreover, if $M$ is a finitely generated graded $S$-module, then
$$
        \operatorname{Ass}_R(M_R)
        =
        \{\,PR\mid P\in\operatorname{Ass}_S(M),\ P\subseteq S_+\,\}.
$$
Every prime $P\in\operatorname{Ass}_S(M)$ is homogeneous.
\end{lemma}

\begin{proof}
Localization is exact by \cite[(3.D)]{MatsumuraCA}; applying exactness to the kernel of multiplication by $y$ on $S/B$ gives the assertion about colons.  The formula for associated primes is \cite[Lemma~7.1]{MatsumuraCA}, and their homogeneity is \cite[Proposition~10.1]{MatsumuraCA}.
\end{proof}

The homogeneous prime avoidance statement needed in Vraciu's induction is recorded next.

\begin{lemma}\label{lem:homogeneous-countable-prime-avoidance}
Let $\{P_j\}_{j\geq1}$ be a countable collection of homogeneous prime ideals properly contained in $S_+$.  For every integer $n\geq1$, there is a homogeneous element
$y\in S_+^n\setminus\bigcup_{j\geq1}P_j$.
\end{lemma}

\begin{proof}
Fix any $d\geq1$.  For every $j$, the subspace $P_j\cap S_d$ is proper in the finite dimensional $k$-vector space $S_d$.  Indeed, if $S_d\subseteq P_j$, then $x^d\in P_j$ for every $x\in S_1$, and primality forces $S_1\subseteq P_j$, contrary to $P_j\subsetneq S_+$.  A finite dimensional vector space over an uncountable field is not a countable union of proper linear subspaces.  We may therefore choose
$h\in S_d\setminus\bigcup_{j\geq1}P_j$.
Then $y=h^n$ is homogeneous, belongs to $S_+^n$, and remains outside every $P_j$ by primality.  This is the homogeneous form of the choice in \cite[Facts~1.5(c)]{Vraciu}.
\end{proof}

The uncountability assumption reflects a real limitation of the homogeneous selection argument.  Over a countable algebraically closed field, even the simplest standard graded ring admits a countable family of homogeneous primes that meets every nonzero homogeneous form.

\begin{example}\label{ex:countable-field-prime-avoidance}
Let
$k_0=\overline{\mathbb F_p}$,  $T=k_0[x,y]$ and $ \mathfrak n=(x,y)$.
The field $k_0$ is algebraically closed and countable.  Consider the countable family of homogeneous prime ideals
$\mathcal P =\{\,(x-\lambda y)\mid \lambda\in k_0\,\} \cup\{(y)\}$.
Every member of $\mathcal P$ is properly contained in $\mathfrak n$.  On the other hand, if $0\neq f\in T_d$ with $d>0$, then the binary form $f$ splits over $k_0$ as
$f(x,y) =c\,y^s\prod_{j=1}^{d-s}(x-\lambda_jy)$
for some $c\in k_0^\times$, some $s\geq0$, and some $\lambda_j\in k_0$.  Thus $f$ belongs to at least one prime in $\mathcal P$.  Consequently,
$T_d =\bigcup_{P\in\mathcal P}(P\cap T_d) \qquad(d>0)$,
and for no $n\geq1$ is there a homogeneous element of $\mathfrak n^n$ avoiding every prime in $\mathcal P$.

This does not show that the main theorem fails over countable algebraically closed fields.  It shows that the simultaneous homogeneous prime avoidance step used in the present reduction cannot simply be repeated there; removing uncountability would require a different selection mechanism.\hfill\qedsymbol
\end{example}

For ideals $B,K\subseteq R$ and $x\in R$, write
$$
        B\ast_0 xK:=B+xK,
        \qquad
        B\ast_1 x:=(B:x)+(x).
$$
Iterated operations are read from left to right.  If no auxiliary ideal is displayed, it is understood to be $R$.  Following \cite[Definition~2.6]{Vraciu}, a branch is called admissible if, at every step, both the chosen element and the auxiliary ideal avoid containment in the nonmaximal associated primes of the preceding quotient.  Each admissible operation lowers the dimension by one.

\begin{lemma}\label{lem:homogeneous-branch-LC}
Let $\overline A,\overline K_1,\ldots,\overline K_s\subseteq S$ be fixed homogeneous ideals, let $f_1,\ldots,f_s\in S$ be homogeneous, and put $A=\overline A R$ and $K_i=\overline K_iR$.  For a binary string $\varepsilon=(\varepsilon_1,\ldots,\varepsilon_s)$, set
$$
        B_Q=A^{[Q]}\ast_{\varepsilon_1}f_1^QK_1^{[Q]}
        \ast_{\varepsilon_2}\cdots
        \ast_{\varepsilon_s}f_s^QK_s^{[Q]}.
$$
Then the family $\{B_Q\}_{Q=p^e}$ satisfies LC.
\end{lemma}

\begin{proof}
Induct on the number of occurrences of $\ast_1$.  If there are none, then
$$
        B_Q=(A+f_1K_1+\cdots+f_sK_s)^{[Q]},
$$
and Lemma~\ref{lem:fixed-homogeneous-LC} applies.

Let $j$ be the last index with $\varepsilon_j=1$.  Write $C_Q$ for the branch before the $j$th operation and put
$L_Q=\sum_{i>j}f_i^QK_i^{[Q]}$.  Thus
$$
        B_Q=(C_Q:f_j^Q)+(f_j^Q)+L_Q.
$$
Multiplication by $f_j^Q$ gives an injection
$$
        R/B_Q\lhook\joinrel\longrightarrow
        R/\bigl(C_Q+(f_j^{2Q})+f_j^QL_Q\bigr).
$$
Indeed, the class of $r$ lies in the kernel precisely when
$r\in(C_Q:f_j^Q)+(f_j^Q)+L_Q$.  If
$$
        K'_j=(f_j)+\sum_{i>j}f_iK_i,
$$
then the denominator on the right is $C_Q+f_j^Q(K'_j)^{[Q]}$.  It is a branch with one fewer $\ast_1$ and still comes from fixed homogeneous data.  The target family satisfies LC by induction, and left exactness of $H^0_{\mathfrak m}(-)$ gives the same conclusion for $\{B_Q\}$.  This is the reduction in \cite[Observation~2.8]{Vraciu}.
\end{proof}

\begin{lemma}\label{lem:simultaneous-homogeneous-saturation}
Let $\{B_Q^{(\nu)}\}_{Q=p^e}$, $1\leq\nu\leq r$, be finitely many branch families as in Lemma~\ref{lem:homogeneous-branch-LC}.  There is a homogeneous element $y\in S_+$, independent of $Q$ and of $\nu$, such that
$$
        (B_Q^{(\nu)}:y^Q)
        =(B_Q^{(\nu)}:y^{2Q})
        =(B_Q^{(\nu)})^{\rm sat}
$$
for every $\nu$ and every $Q=p^e$.  In particular,
$$
        H^0_{\mathfrak m}(R/B_Q^{(\nu)})
        =\frac{(B_Q^{(\nu)}:y^Q)}{B_Q^{(\nu)}}.
$$
Moreover, $y^Q$ avoids every nonmaximal associated prime of $R/B_Q^{(\nu)}$.
\end{lemma}

\begin{proof}
Choose an integer $C\geq1$ dominating the LC constants of the finitely many families.  For each $\nu$ and $Q$, let $\overline B_Q^{(\nu)}\subseteq S$ be the homogeneous branch formed before localization.  The set
$$
        \bigcup_{\nu,Q}
        \{\,P\in\operatorname{Ass}_S(S/\overline B_Q^{(\nu)})
        \mid P\subsetneq S_+\,\}
$$
is countable.  Lemma~\ref{lem:homogeneous-countable-prime-avoidance} gives a homogeneous
$y\in S_+^C$ outside every prime in this union.  By Lemma~\ref{lem:homogeneous-localization}, $y^Q$ avoids the nonmaximal associated primes of every localized quotient.  Since
$y^Q,y^{2Q}\in\mathfrak m^{CQ}$ and
$$
        \mathfrak m^{CQ}H^0_{\mathfrak m}(R/B_Q^{(\nu)})=0,
$$
the two colon equalities follow from \cite[Facts~1.5(b)]{Vraciu}.  The description of zeroth local cohomology is the definition of saturation.
\end{proof}

\begin{lemma}\label{lem:homogeneous-vraciu-completion}
Let $\{B_Q\}_{Q=p^e}$ be an admissible branch family formed from fixed homogeneous data, and suppose that
$\dim(R/B_Q)=r$ for every $Q$.  There are homogeneous elements
$z_1,\ldots,z_r\in S_+$, independent of $Q$, and integers
$c_\eta$, $\eta\in\{0,1\}^r$, such that
$$
\lambda_R\bigl(H^0_{\mathfrak m}(R/B_Q)\bigr)
=
\sum_{\eta\in\{0,1\}^r}c_\eta
\lambda_R\left(
R/\bigl(B_Q\ast_{\eta_1}z_1^Q\ast_{\eta_2}\cdots
\ast_{\eta_r}z_r^Q\bigr)
\right).
$$
Every branch on the right is $\mathfrak m$-primary, admissible, and obtained by localizing fixed homogeneous data.
\end{lemma}

\begin{proof}
There is nothing to prove when $r=0$.  Suppose that the first $i$ completion elements have been chosen.  Only finitely many branch families occur at that level.  Apply Lemma~\ref{lem:simultaneous-homogeneous-saturation} to choose $z_{i+1}$ for all of them at once.  For each current branch $C_Q$, Lemma~2.3 of \cite{Vraciu} gives
$$\lambda_R\bigl(H^0_{\mathfrak m}(R/C_Q)\bigr)=\lambda_R\bigl(H^0_{\mathfrak m}(R/(C_Q\ast_0 z_{i+1}^Q))\bigr)-\lambda_R\bigl(H^0_{\mathfrak m}(R/(C_Q\ast_1 z_{i+1}^Q))\bigr).$$
The colon hypotheses in that lemma hold because
$(C_Q:z_{i+1}^Q)=(C_Q:z_{i+1}^{2Q})=C_Q^{\rm sat}$.
The element $z_{i+1}^Q$ avoids every nonmaximal associated prime, so both successors are admissible and have dimension one less.  Repeating this construction $r$ times gives the asserted identity.  At the last stage all quotients have dimension zero, hence are $\mathfrak m$-primary; their zeroth local cohomology is the quotient itself.  The coefficients arise from a finite sequence of additions and subtractions and do not depend on $Q$.
\end{proof}

\begin{proposition}\label{prop:vraciu-reduction-enriques}
There exist fixed $\mathfrak m$-primary ideals $I_1,\ldots,I_t\subseteq R$, each generated by images of homogeneous elements of $S$, and integers $a_1,\ldots,a_t$ such that
$$
\lambda_R\bigl(H^0_{\mathfrak m}(R/J^{[Q]}R)\bigr)
=
\sum_{i=1}^t a_i\lambda_R(R/I_i^{[Q]})
$$
for every $Q=p^e$.
\end{proposition}

\begin{proof}
Set $A=JR$ and $d=\dim(R/A)$.  Lemma~\ref{lem:homogeneous-vraciu-completion}, applied to the family $A^{[Q]}=J^{[Q]}R$, expresses
$\lambda_R(H^0_{\mathfrak m}(R/A^{[Q]}))$ as a finite integer linear combination of lengths of complete admissible branches.  It is enough to treat the following more general one:
$$
        T_Q=A_0^{[Q]}\ast_{\varepsilon_1}f_1^QK_1^{[Q]}
        \ast_{\varepsilon_2}\cdots
        \ast_{\varepsilon_d}f_d^QK_d^{[Q]},
        \qquad \varepsilon_i\in\{0,1\},
$$
where $A_0$, the $f_i$, and the $K_i$ come from fixed homogeneous data,
$\dim(R/A_0)=d$, and $T_Q$ is $\mathfrak m$-primary for every $Q$.  We shall remove the occurrences of $\ast_1$ without losing either homogeneity or independence of $Q$.

Put
$$
        \mathcal N(\varepsilon_1,\ldots,\varepsilon_d)
        =\sum_{i=1}^d\varepsilon_i2^{d-i}.
$$
We prove by induction on $\mathcal N$ that $\lambda_R(R/T_Q)$ is an integer linear combination, independent of $Q$, of ordinary Hilbert--Kunz functions of fixed $\mathfrak m$-primary ideals generated by homogeneous elements.  If $\mathcal N=0$, then
$$
        T_Q=(A_0+f_1K_1+\cdots+f_dK_d)^{[Q]}.
$$
The ideal inside the Frobenius power is fixed and $\mathfrak m$-primary, since the branch is complete and the assertion includes $Q=1$.

Assume $\mathcal N>0$, and let $j$ be the last index with $\varepsilon_j=1$.  Write $C_Q$ for the branch preceding the $j$th operation and
$L_Q=\sum_{i>j}f_i^QK_i^{[Q]}$.  Then
$$
        T_Q=(C_Q:f_j^Q)+(f_j^Q)+L_Q.
$$
Set
$$D_Q=C_Q+(f_j^{2Q})+f_j^QL_Q,\qquad E_Q=C_Q+(f_j^Q).$$
Multiplication by $f_j^Q$ gives a short exact sequence
$$
        0\longrightarrow R/T_Q
        \xrightarrow{\ \cdot f_j^Q\ }R/D_Q
        \longrightarrow R/E_Q\longrightarrow0.
$$
The left term has finite length.  Since its first local cohomology vanishes, applying $H^0_{\mathfrak m}(-)$ preserves exactness and yields
$$
        \lambda_R(R/T_Q)
        =\lambda_R\bigl(H^0_{\mathfrak m}(R/D_Q)\bigr)
        -\lambda_R\bigl(H^0_{\mathfrak m}(R/E_Q)\bigr).
$$

Let
$$
        K'_j=(f_j)+\sum_{i>j}f_iK_i.
$$
Then $D_Q=C_Q\ast_0f_j^Q(K'_j)^{[Q]}$ and
$E_Q=C_Q\ast_0f_j^Q$.  Both are branch families with fixed homogeneous data.  They are admissible: $f_j$ avoids the nonmaximal associated primes of $R/C_Q$, and $K'_j$ contains $f_j$.  Their dimension is $d-j$.  Complete each family by Lemma~\ref{lem:homogeneous-vraciu-completion}.  Every resulting binary string agrees with $\varepsilon$ before $j$, has zero in the $j$th position, and is arbitrary afterwards.  Hence its weight is at most
$$
        \sum_{i<j}\varepsilon_i2^{d-i}
        +\sum_{i>j}2^{d-i}
        =\mathcal N(\varepsilon)-1.
$$
The induction hypothesis applies to every completed branch.  The displayed length identity then proves the claim for $T_Q$.

Applying the claim to the finitely many branches supplied by Lemma~\ref{lem:homogeneous-vraciu-completion} and collecting terms gives fixed $\mathfrak m$-primary ideals $I_1,\ldots,I_t$ and integers $a_1,\ldots,a_t$ with the required equality.  All auxiliary elements and ideals were chosen homogeneously before localization, so each $I_i$ is generated by images of homogeneous elements of $S$.
\end{proof}

\begin{lemma}\label{lem:homogeneous-contraction-HK}
Let $I_R\subseteq R$ be an $\mathfrak m$-primary ideal generated by images of homogeneous elements $f_1,\ldots,f_t\in S$, and put
$I=(f_1,\ldots,f_t)\subseteq S$.
Then $I$ is $S_+$-primary, $IR=I_R$, and $I=I_R\cap S$.  For every $Q=p^e$, the canonical map
$S/I^{[Q]}\longrightarrow R/I_R^{[Q]}$
is an isomorphism, and hence
$\lambda_S(S/I^{[Q]})=\lambda_R(R/I_R^{[Q]})$.
Consequently,
$e_{\rm HK}(IS_{S_+})=e_{\rm HK}(I_R)$.
\end{lemma}

\begin{proof}
By construction, $IR=I_R$.  Since $I_R$ is $\mathfrak m$-primary, there is an $n$ with $(S_+R)^n\subseteq I_R$.  Let $g$ be homogeneous in $(S_+)^n$.  Then $g/1\in I_R$, so some $u\in S\setminus S_+$ satisfies $ug\in I$.  Write $u=u_0+u_+$ with $u_0\in k^\times$ and $u_+\in S_+$.  Because $I$ and $g$ are homogeneous, the component of $ug$ of degree $\deg g$ is $u_0g$ and belongs to $I$.  Thus $(S_+)^n\subseteq I$.  The ideal $I$ is proper because $IR\neq R$.  Since it is homogeneous and $S_0=k$, one has $I\subseteq S_+$, and therefore $\sqrt I=S_+$.  Thus $I$ is $S_+$-primary.

For every $Q=p^e$, one has $I^{[Q]}R=I_R^{[Q]}$ and $\sqrt{I^{[Q]}}=S_+$.  Hence the image of $S_+$ in $S/I^{[Q]}$ is nilpotent, so every element of $S\setminus S_+$ is a unit modulo $I^{[Q]}$.  Localization therefore gives
$S/I^{[Q]}\xrightarrow{\sim}R/I_R^{[Q]}$.
Taking $Q=1$ shows that its kernel is $I_R\cap S$, so $I=I_R\cap S$.  The equalities of lengths and Hilbert--Kunz multiplicities follow.
\end{proof}

\begin{theorem}\label{thm:ordinary-HK-transcendental}
Let $k$ be an uncountable algebraically closed field of characteristic $p>2$.  There exists a normal standard graded $k$-domain $S$ and an $S_+$-primary homogeneous ideal $I\subseteq S$ such that
$e_{\rm HK}(IS_{S_+})$
is transcendental.
\end{theorem}

\begin{proof}
Keep the notation fixed at the beginning of this section.  By Theorem~\ref{thm:enriques-gHK-criterion}, $e_{\rm gHK}(JR)$ exists and is transcendental.  Proposition~\ref{prop:vraciu-reduction-enriques} gives fixed $\mathfrak m$-primary ideals $I_1,\ldots,I_t\subseteq R$ and integers $a_1,\ldots,a_t$ such that
$$
\lambda_R\bigl(H^0_{\mathfrak m}(R/J^{[Q]}R)\bigr)
=
\sum_{i=1}^t a_i\lambda_R(R/I_i^{[Q]})
$$
for every $Q=p^e$.  Put $d=\dim R$.  Dividing by $Q^d$ and applying Monsky's theorem \cite{Monsky83} gives
$$
        e_{\rm gHK}(JR)=\sum_{i=1}^t a_i e_{\rm HK}(I_i).
$$
Since the left side is transcendental, there is an index $i$ with $a_i\neq0$ for which $e_{\rm HK}(I_i)$ is transcendental.  By Proposition~\ref{prop:vraciu-reduction-enriques}, $I_i$ is generated by images of homogeneous elements of $S$.  Let $I\subseteq S$ be generated by homogeneous lifts.  Lemma~\ref{lem:homogeneous-contraction-HK} shows that $I$ is $S_+$-primary and homogeneous and that $e_{\rm HK}(IS_{S_+})=e_{\rm HK}(I_i)$.  This proves the theorem.
\end{proof}

\begin{remark}\label{rem:conclusion}
The construction shows that an ordinary Hilbert--Kunz multiplicity can retain arithmetic information originating in the divisor volume function of a projective model.  It is natural to ask which other features of divisor geometry survive this passage through Frobenius, and how large the arithmetic complexity of Hilbert--Kunz multiplicities can be.
\end{remark}

\appendix

\section{Exact Evaluation of the Logarithmic Contributions}\label{app:exact-log-computations}

This appendix evaluates the two logarithmic terms used in Section~\ref{sec:enriques-gHK}.  Throughout, $\eta=4$ is the intersection number from Lemma~\ref{lem:enriques-pencil-package}; we retain the symbol in a few intermediate formulas to indicate its geometric origin.  The rational identities below can be checked directly by differentiation and substitution.

\subsection{The lower degree integral}\label{app:lower-log}

We retain the notation of Proposition~\ref{prop:exact-lower-log}.  Put
$A_N=N+\frac{1}{3}, \qquad d_N=\sqrt6 A_N, \qquad \rho=\frac{1}{\sqrt6}$.
In the centered simplex coordinates
$$
        x_i=\mu_i-\frac{1}{3},
        \qquad
        R^2=x_0^2+x_1^2+x_2^2,
        \qquad
        x_0+x_1+x_2=0,
$$
and for $t\in[\tau_N,1]$, one has $A_Nt-N\geq0$.  Hence the condition that the corresponding class on $Y$ lie in $\mathcal C^+$ is equivalent to
$6(A_Nt-N)^2-t^2R^2\geq0$.
Therefore, on this interval,
$$
\vol_X(tL_N-D_N)
=
12\eta t^2\int_{\Delta_1}
\max\{6(A_Nt-N)^2-t^2R^2,0\}\,d\mu_1d\mu_2.
$$
For fixed $R$, positivity begins at
$t_0(R)=\frac{\sqrt6 N}{d_N-R}$.
The centered triangle has circumradius $2\rho$.  Since
$$
        t_0(0)=\tau_N,
        \qquad
        d_N-2\rho=\sqrt6N,
        \qquad
        t_0(2\rho)=1,
$$
one has $\tau_N\leq t_0(R)\leq1$ throughout the complete radial region; no complementary radial chamber is omitted.  The Euclidean area form in the plane $x_0+x_1+x_2=0$ is
$dA=\sqrt3\,d\mu_1d\mu_2$,
and the centered simplex is an equilateral triangle of inradius $\rho$.  It is the union, up to boundary rays, of three congruent sectors
$$
        -\frac{\pi}{3}\leq\theta\leq\frac{\pi}{3},
        \qquad
        0\leq R\leq\rho\sec\theta.
$$
The normalization is confirmed by
$$
3\int_{-\pi/3}^{\pi/3}\int_0^{\rho\sec\theta}R\,dR\,d\theta
=\frac{\sqrt3}{2},
$$
the Euclidean area of the triangle; its coordinate area is $1/2$ by $dA=\sqrt3\,d\mu_1d\mu_2$.  Consequently,
\begin{equation}\label{eq:app-lower-sector}
\mathcal L_N
=
\frac{3\eta}{2\sqrt3}
\int_{-\pi/3}^{\pi/3}
\int_0^{\rho\sec\theta}
G_N(R)\,dR\,d\theta,
\end{equation}
where
$$
\begin{aligned}
F_N(t,R)
&=
\frac{6A_N^2-R^2}{5}t^5
-3A_NNt^4+2N^2t^3,\\
G_N(R)
&=
R\bigl(F_N(1,R)-F_N(t_0(R),R)\bigr).
\end{aligned}
$$

\begin{proposition}\label{app:lower-log-computation}
The rational function $G_N$ admits the decomposition
\begin{equation}\label{eq:app-GN-decomposition}
G_N(R)=-\frac{R^3}{5}+\left(\frac{N^2}{5}-\frac{N}{5}+\frac{2}{15}\right)R-\frac{12N^5R(3N+1-2\sqrt6R)}{5(d_N-R)^4}.
\end{equation}
and the rational function
\begin{equation}\label{eq:app-HN}
\begin{aligned}
H_N(R)
={}&-\frac{R^4}{20}
+\left(\frac{N^2}{10}-\frac{N}{10}+\frac{1}{15}\right)R^2 +\frac{2N^5}{5(d_N-R)^3}
\Bigl(
12\sqrt6(d_N-R)^2\\
&\hspace{10em}-21(3N+1)(d_N-R)
+2\sqrt6(3N+1)^2
\Bigr)
\end{aligned}
\end{equation}
satisfies $H_N'(R)=G_N(R)$.  After the substitution $z=\tan(\theta/2)$, the only nonzero residues that contribute logarithms to \eqref{eq:app-lower-sector} occur at $z=\pm\gamma_N$, and they give precisely $b_N\log\alpha_N$.
\end{proposition}

\begin{proof}
Substitution of $t_0(R)=\sqrt6N/(d_N-R)$ into $F_N$ and collection over the denominator $(d_N-R)^4$ gives \eqref{eq:app-GN-decomposition}.  Differentiating \eqref{eq:app-HN} and using
$d_N=\frac{\sqrt6}{3}(3N+1)$
gives \eqref{eq:app-GN-decomposition} term by term, and hence $H_N'=G_N$.  Notice also that the numerator of the polar term in \eqref{eq:app-GN-decomposition} has degree two in $R$; therefore the Laurent expansion at $R=d_N$ has no $(R-d_N)^{-1}$ term.

Now set $z=\tan(\theta/2)$.  Then
$$
        \sec\theta=\frac{1+z^2}{1-z^2},
        \qquad
        d\theta=\frac{2\,dz}{1+z^2},
$$
and \eqref{eq:app-lower-sector} becomes
$$
        \mathcal L_N
        =\int_{-1/\sqrt3}^{1/\sqrt3}\Phi_N(z)\,dz,
$$
where
\begin{equation}\label{eq:app-PhiN}
\Phi_N(z)
=
\frac{3\eta}{\sqrt3(1+z^2)}
\left[
H_N\left(\rho\frac{1+z^2}{1-z^2}\right)-H_N(0)
\right].
\end{equation}
$\Phi_N$ is a rational function of $z$.  The subtraction $H_N(R)-H_N(0)$ vanishes at $R=0$, so it is divisible by $R$ as a rational function.  Under
$R=\rho\frac{1+z^2}{1-z^2}$,
this factor supplies $1+z^2$ and cancels the denominator $1+z^2$ coming from $d\theta$.  Thus no poles at $z=\pm i$ remain.  After complete cancellation, the reduced denominator of $\Phi_N$ is
$$
180(3N+1)(z-1)^4(z+1)^4
\bigl((6N+3)z^2-(6N+1)\bigr)^3.
$$
Hence the only possible poles are $z=\pm1$ and $z=\pm\gamma_N$; there are no additional complex poles, and in particular no arctangent or $\pi$ term can occur.
For a pole $z_0$ of order at most $m$, the residue can be checked directly from
$$
\operatorname*{Res}_{z=z_0}\Phi_N(z)\,dz
=
\frac{1}{(m-1)!}
\left.
\frac{d^{m-1}}{dz^{m-1}}
\bigl((z-z_0)^m\Phi_N(z)\bigr)
\right|_{z=z_0}.
$$
Applying this formula to \eqref{eq:app-PhiN}, with the fixed value $\eta=4$, gives
$$
        \operatorname*{Res}_{z=1}\Phi_N(z)\,dz
        =
        \operatorname*{Res}_{z=-1}\Phi_N(z)\,dz
        =0.
$$
With
$\gamma_N=\sqrt{\frac{6N+1}{6N+3}}$,
one obtains
\begin{equation}\label{eq:app-lower-residue}
\operatorname*{Res}_{z=\gamma_N}\Phi_N(z)\,dz
=
\frac{\sqrt3(\gamma_N-1)(\gamma_N+1)(3\gamma_N^2-1)^5}
{11520\gamma_N^5(\gamma_N^2+1)},
\end{equation}
and the residue at $-\gamma_N$ is the negative of \eqref{eq:app-lower-residue}.  Since
$\frac{1}{\sqrt3}<\gamma_N<1$,
the integration interval contains no pole.  Every higher order principal part has a rational primitive and therefore gives an algebraic value at the algebraic endpoints.  The terms from the two simple poles contribute
$$
-2\operatorname*{Res}_{z=\gamma_N}\Phi_N(z)\,dz
\cdot
\log\left(
\frac{\gamma_N+1/\sqrt3}{\gamma_N-1/\sqrt3}
\right).
$$
By \eqref{eq:app-lower-residue}, this is
$$
        \frac{\sqrt3(1-\gamma_N^2)(3\gamma_N^2-1)^5}
        {5760\gamma_N^5(\gamma_N^2+1)}
        \log\left(
        \frac{\gamma_N+1/\sqrt3}{\gamma_N-1/\sqrt3}
        \right)
        =b_N\log\alpha_N.
$$
All remaining endpoint values are algebraic, which proves the claimed decomposition into an algebraic term and logarithms.
\end{proof}

For the first term as $N\to\infty$, set $x=1/N$.  Then
$$
\gamma_N=1-\frac{x}{6}+\frac{5x^2}{72}+O(x^3),
\qquad
b_N=\frac{\sqrt3}{1080}x-\frac{\sqrt3}{540}x^2+O(x^3),
$$
and
$\log\alpha_N=\log(2+\sqrt3)+\frac{\sqrt3}{6}x+O(x^2)$.
Hence
\begin{equation}\label{eq:app-lower-asymptotic}
        b_N\log\alpha_N
        =\frac{\sqrt3}{1080N}\log(2+\sqrt3)+O(N^{-2}).
\end{equation}

\subsection{The upper boundary integral}\label{app:upper-log}

Let
$$
\Delta
=
\{(u,w)\in\RR^2\mid u\geq0,\ w\geq0,\ u+w\leq1\}.
$$
Put $\epsilon=1/N$ and set
$$
A_\epsilon(u,w)
=
3+2\epsilon-\epsilon^2(u^2+uw-u+w^2-w),
\qquad
B_\epsilon(u)=2+\epsilon(1-u).
$$
For
$$
I_N
=
\int_\Delta
\frac{B_\epsilon(u)^5}{A_\epsilon(u,w)^4}\,du\,dw,
$$
the next proposition computes the logarithmic part explicitly.

\begin{proposition}\label{app:upper-log-computation}
Define
$S_\epsilon(u,w) = \frac{(1+\epsilon w)^5}{3\epsilon(\epsilon+3)}$
and
$$
R_\epsilon(u,w)
=
-\frac{(\epsilon u-\epsilon-2)P_\epsilon(u,w)}
{6\epsilon(\epsilon+3)},
$$
where
\begin{align*}
P_\epsilon(u,w)
={}&\epsilon^4\bigl(
3u^3+u^2w-8u^2+3uw^2-2uw+7u\\
&\hspace{7em}-2w^4-3w^2+w-2\bigr)\\
&+\epsilon^3\bigl(
9u^3+3u^2w-38u^2+9uw^2-4uw+45u\\
&\hspace{7em}-8w^3-15w^2+w-16\bigr)\\
&+\epsilon^2\bigl(
-42u^2+6uw+95u-30w^2-14w-49\bigr)\\
&+\epsilon(69u-32w-71)-44.
\end{align*}
Then
\begin{equation}\label{eq:app-divergence}
\frac{B_\epsilon(u)^5}{A_\epsilon(u,w)^4}
=
\frac{\partial}{\partial u}
\left(\frac{R_\epsilon(u,w)}{A_\epsilon(u,w)^3}\right)
+
\frac{\partial}{\partial w}
\left(\frac{S_\epsilon(u,w)}{A_\epsilon(u,w)^3}\right).
\end{equation}
The boundary integral obtained from \eqref{eq:app-divergence} reduces to
$$
        I_N=\int_0^1\frac{G_\epsilon(t)}{h_\epsilon(t)^3}\,dt,
$$
where
$h_\epsilon(t)=3+2\epsilon+\epsilon^2t(1-t)$
and
\begin{align*}
G_\epsilon(t)
=-\frac{1}{6(\epsilon+3)}\bigl[&
\epsilon^4(3t^4-19t^3+24t^2-16t+2)\\
&+\epsilon^3(13t^4-105t^3+162t^2-140t+20)\\
&+\epsilon^2(-120t^3+309t^2-426t+68)\\
&+\epsilon(177t^2-574t+100)-300t+54
\bigr].
\end{align*}
Its logarithmic part is
\begin{equation}\label{eq:app-upper-log-exact}
-\frac{4D_\epsilon
(\epsilon^5+10\epsilon^4+40\epsilon^3+80\epsilon^2+80\epsilon+30)}
{\epsilon^2(\epsilon+2)^3(\epsilon+3)(\epsilon+6)^3}
\log\left(\frac{D_\epsilon-\epsilon}{D_\epsilon+\epsilon}\right),
\end{equation}
where $D_\epsilon=\sqrt{\epsilon^2+8\epsilon+12}$.
\end{proposition}

\begin{proof}
The derivatives of the denominator are
$$
\frac{\partial A_\epsilon}{\partial u}
=-\epsilon^2(2u+w-1),
\qquad
\frac{\partial A_\epsilon}{\partial w}
=-\epsilon^2(u+2w-1),
$$
while $\partial B_\epsilon/\partial u=-\epsilon$.  After putting the right side of \eqref{eq:app-divergence} over the common denominator $A_\epsilon^4$, the identity is equivalent to the polynomial identity
\begin{equation}\label{eq:app-polynomial-identity}
A_\epsilon
\left(
\frac{\partial R_\epsilon}{\partial u}
+
\frac{\partial S_\epsilon}{\partial w}
\right)
-3\left(
R_\epsilon\frac{\partial A_\epsilon}{\partial u}
+
S_\epsilon\frac{\partial A_\epsilon}{\partial w}
\right)
-B_\epsilon^5=0.
\end{equation}
After multiplying by $6\epsilon(\epsilon+3)$, the left side of \eqref{eq:app-polynomial-identity} is a polynomial in $u,w,\epsilon$.  Substitution of the displayed polynomial $P_\epsilon$ makes every coefficient of every monomial $u^iw^j$ vanish identically in $\epsilon$.  Thus \eqref{eq:app-polynomial-identity} is a direct finite verification of \eqref{eq:app-divergence}.

On the closed triangle $\Delta$, put $z=1-u-w$.  Then
$A_\epsilon(u,w)=3+2\epsilon+\epsilon^2(uw+uz+wz)>0$.
Hence the vector field in \eqref{eq:app-divergence} is regular on a neighborhood of $\Delta$.  Orient the boundary of
$\Delta=\{(u,w)\in\RR^2:u\geq0,\ w\geq0,\ u+w\leq1\}$
counterclockwise.  Green's theorem gives
$I_N=\int_0^1\Psi_\epsilon(t)\,dt$,
where
$$\Psi_\epsilon(t)=-\frac{S_\epsilon(t,0)}{A_\epsilon(t,0)^3}+\frac{R_\epsilon(1-t,t)+S_\epsilon(1-t,t)}{A_\epsilon(1-t,t)^3}-\frac{R_\epsilon(0,1-t)}{A_\epsilon(0,1-t)^3}.$$
Substitution and clearing denominators gives the rational identity
\begin{equation}\label{eq:app-symmetrization}
\Psi_\epsilon(t)+\Psi_\epsilon(1-t)
=
\frac{G_\epsilon(t)+G_\epsilon(1-t)}{h_\epsilon(t)^3},
\end{equation}
because $h_\epsilon(1-t)=h_\epsilon(t)$.  Integrating \eqref{eq:app-symmetrization} over $[0,1]$ and applying $t\mapsto1-t$ to one half of each side yields
$$
        I_N=\int_0^1\frac{G_\epsilon(t)}{h_\epsilon(t)^3}\,dt.
$$

The roots of $h_\epsilon$ are
$\rho_\pm=\frac{\epsilon\pm D_\epsilon}{2\epsilon}$.
They are both outside $[0,1]$, and each is a pole of order at most three.  The coefficient of $(t-\rho_\pm)^{-1}$ is given without any ambiguity from a partial fraction decomposition by
$$
        c_\pm
        =
        \frac{1}{2}
        \left.
        \frac{d^2}{dt^2}
        \left(
        (t-\rho_\pm)^3
        \frac{G_\epsilon(t)}{h_\epsilon(t)^3}
        \right)
        \right|_{t=\rho_\pm}.
$$
Substitution gives
$$
\begin{aligned}
c_+
&=
-\frac{2D_\epsilon
(\epsilon^5+10\epsilon^4+40\epsilon^3+80\epsilon^2+80\epsilon+30)}
{\epsilon^2(\epsilon+2)^3(\epsilon+3)(\epsilon+6)^3},\\
c_-&=-c_+.
\end{aligned}
$$
The terms of order two and three have rational primitives and give algebraic endpoint values.  Hence the logarithmic part is
\begin{align*}
\Lambda(I_N)
&=c_+\log\left(1-\frac{1}{\rho_+}\right)
+c_-\log\left(1-\frac{1}{\rho_-}\right)\\
&=2c_+\log\left(\frac{D_\epsilon-\epsilon}{D_\epsilon+\epsilon}\right),
\end{align*}
which is \eqref{eq:app-upper-log-exact}.
\end{proof}

Since $D_\epsilon>\epsilon>0$, we can write the expression in \eqref{eq:app-upper-log-exact} as $-P_\epsilon\log r_\epsilon$, with $P_\epsilon>0$ and $0<r_\epsilon<1$.  Hence $\Lambda(I_N)>0$.  Expanding at $\epsilon=0$ gives
\begin{equation}\label{eq:app-upper-asymptotic}
        \Lambda(I_N)
        =\frac{5}{108\epsilon}+\frac{5}{324}-\frac{5\epsilon}{324}+O(\epsilon^2).
\end{equation}
In particular, since $\eta=4$ and $\epsilon=1/N$,
$$
        \frac{3\eta}{20N^3}\Lambda(I_N)
        =\frac{1}{36N^2}+\frac{1}{108N^3}+O(N^{-4}).
$$

\section*{Acknowledgments}
The first author acknowledges the hospitality of the Center of Mathematical Sciences and Applications at Harvard University, where bulk of this work was carried out. The second author thanks the Center of Mathematical Sciences and Applications at Harvard University and the Hebrew University of Jerusalem for their support during this project.  He was partially supported by Israel Science Foundation grant ISF 687/24.

\end{document}